\documentclass[11pt]{amsart}

\usepackage{mathrsfs}
\usepackage{amsfonts}
\usepackage{amssymb}
\usepackage{amsxtra}
\usepackage{dsfont}
\usepackage{color}
\usepackage[english,polish]{babel}
\usepackage[shortlabels]{enumitem}

\usepackage[compress, sort]{cite}

\usepackage{newtxmath}
\usepackage{newtxtext}

\usepackage[margin=2.5cm, centering]{geometry}
\usepackage[colorlinks,citecolor=blue,urlcolor=blue,bookmarks=true]{hyperref}
\hypersetup{
pdfpagemode=UseNone,
pdfstartview=FitH,
pdfdisplaydoctitle=true,
pdfborder={0 0 0}, % No link borders
pdftitle={About essential spectra of unbounded Jacobi matrices},
pdfauthor={Grzegorz Świderski and Bartosz Trojan},
pdflang=en-US
}

\newcommand{\CC}{\mathbb{C}}
\newcommand{\ZZ}{\mathbb{Z}}

\newcommand{\NN}{\mathbb{N}}
\newcommand{\RR}{\mathbb{R}}
\newcommand{\PP}{\mathbb{P}}

\newcommand{\calR}{\mathcal{R}}
\newcommand{\calC}{\mathcal{C}}

\newcommand{\calO}{\mathcal{O}}

\newcommand{\calB}{\mathcal{B}}

\newcommand{\calX}{\mathcal{X}}

\newcommand{\calD}{\mathcal{D}}

\newcommand{\frakB}{\mathfrak{B}}
\newcommand{\frakX}{\mathfrak{X}}

\newcommand{\frakp}{\mathfrak{p}}
\newcommand{\frakf}{\mathfrak{f}}

\newcommand{\Id}{\operatorname{Id}}

\newcommand{\sign}[1]{\operatorname{sign}({#1})}

\newcommand{\pl}[1]{\foreignlanguage{polish}{#1}}

\newcommand{\abs}[1]{\lvert {#1} \rvert}
\newcommand{\sprod}[2]{\langle {#1}, {#2} \rangle}

\newcommand{\tr}{\operatorname{tr}}

\newcommand{\GL}{\operatorname{GL}}
\newcommand{\discr}{\operatorname{discr}}

\newcommand{\Mat}{\operatorname{Mat}}

\newcommand{\sigmaEss}{\sigma_{\mathrm{ess}}}

\newcommand{\sigmaAC}{\sigma_{\mathrm{ac}}}
\newcommand{\sigmaS}{\sigma_{\mathrm{sing}}}

\newcommand{\ue}{\textrm{e}}
\newcommand{\supp}{\operatornamewithlimits{supp}}

\newtheorem{theorem}{Theorem}[section]

\newtheorem{proposition}[theorem]{Proposition}
\newtheorem{lemma}[theorem]{Lemma}
\newtheorem{corollary}[theorem]{Corollary}
\newtheorem{claim}[theorem]{Claim}

\theoremstyle{plain}
\newcounter{thm}

\newtheorem{main_theorem}[thm]{Theorem}

\theoremstyle{definition}
\newtheorem{example}[theorem]{Example}
\newtheorem{remark}[theorem]{Remark}
\newtheorem{definition}[theorem]{Definition}

\usepackage[utf8]{inputenc}

\numberwithin{equation}{section}

\title[Essential spectra of unbounded Jacobi matrices]
{About essential spectra of unbounded Jacobi matrices}
\author{Grzegorz Świderski}
\address{
	\pl{
		Grzegorz \'Swiderski \\
		Department of Mathematics \\
		KU Leuven \\
		Celestijnenlaan 200B box 2400 \\
		BE-3001 Leuven \\
		Belgium \&
		Mathematical Institute \\
		University of Wroc\l{}aw \\
		pl. Grunwaldzki 2/4 \\
		50-384 Wroc\l{}aw \\
		Poland
	}
}
\email{grzegorz.swiderski@kuleuven.be}

\author{Bartosz Trojan}
\address{
	\pl{
		Bartosz Trojan\\
        Instytut Matematyczny\\
        Polskiej Akademii Nauk\\
        ul. \'Sniadeckich 8\\
        00-696 Warszawa\\
        Poland}
}
\email{btrojan@impan.pl}



\begin{document}
\selectlanguage{english}

\begin{abstract}
We study spectral properties of unbounded Jacobi matrices with periodically modulated or blended entries. 
Our approach is based on uniform asymptotic analysis of generalized eigenvectors. We determine when the
studied operators are self-adjoint. We identify regions where the point spectrum has no accumulation points. 
This allows us to completely describe the essential spectrum of these operators.
\end{abstract}
\maketitle

\section{Introduction}
Consider two sequences $a = (a_n : n \in \NN_0)$ and $b = (b_n : n \in \NN_0)$ such that $a_n > 0$ and $b_n \in \RR$
for all $n \geq 0$. Let $A$ be the closure in $\ell^2(\NN_0)$ of the operator acting on sequences having finite support
by the matrix
\[
	\begin{pmatrix}
		b_0 & a_0 & 0   & 0      &\ldots \\
		a_0 & b_1 & a_1 & 0       & \ldots \\
		0   & a_1 & b_2 & a_2     & \ldots \\
		0   & 0   & a_2 & b_3   &  \\
		\vdots & \vdots & \vdots  &  & \ddots
	\end{pmatrix}.
\]
The operator $A$ is called \emph{Jacobi matrix}. Recall that $\ell^2(\NN_0)$ is the Hilbert space
of square summable complex valued sequences with the scalar product
\[
	\sprod{x}{y}_{\ell^2(\NN_0)} = \sum_{n=0}^\infty x_n \overline{y_n}.
\]
The most throughly studied are bounded Jacobi matrices, see e.g. \cite{Simon2010Book}. Let us remind that the Jacobi
matrix $A$ is bounded if and only if the sequences $a$ and $b$ are bounded. 
In this article we are exclusively interested in \emph{unbounded} Jacobi matrices. We shall consider 
two classes: periodically modulated and periodically blended. The first class has been introduced in
\cite{JanasNaboko2002} and systematically studied since then. To be precise, let $N$ be a positive
integer. We say that $A$ has \emph{$N$-periodically modulated} entries if there are two $N$-periodic
sequences $(\alpha_n : n \in \ZZ)$ and $(\beta_n : n \in \ZZ)$ of positive and real numbers, respectively,
such that
\begin{enumerate}[leftmargin=2em, label=\alph*)]
	\item
	$\begin{aligned}[b]
	\lim_{n \to \infty} a_n = \infty
	\end{aligned},$
	\item
	$\begin{aligned}[b]
	\lim_{n \to \infty} \bigg| \frac{a_{n-1}}{a_n} - \frac{\alpha_{n-1}}{\alpha_n} \bigg| = 0
	\end{aligned},$
	\item
	$\begin{aligned}[b]
	\lim_{n \to \infty} \bigg| \frac{b_n}{a_n} - \frac{\beta_n}{\alpha_n} \bigg| = 0
	\end{aligned}.$
\end{enumerate}
This class contains sequences one can find in many applications. It is also rich enough to allow building an intuition
about the general case. In particular, in this class there are examples of Jacobi
matrices with purely absolutely continuous spectrum filling the whole real line (see \cite{JanasNaboko2002, 
PeriodicI, PeriodicII, SwiderskiTrojan2019, JanasNaboko2001}), having a bounded gap in absolutely continuous 
spectrum (see \cite{Dombrowski2004, Dombrowski2009, DombrowskiJanasMoszynskiEtAl2004, DombrowskiPedersen2002a, 
DombrowskiPedersen2002, JanasMoszynski2003, JanasNabokoStolz2004, Sahbani2016, Janas2012}), having absolutely
continuous spectrum on the half-line (see \cite{Damanik2007, Janas2001, Janas2009, Motyka2015, Naboko2009,
Naboko2010, Simonov2007, DombrowskiPedersen1995, Naboko2019}), having purely singular continuous spectral measure
with explicit Hausdorff dimension (see \cite{Breuer2010}), having a dense point spectrum on the real line
(see \cite{Breuer2010}), and having an empty essential spectrum (see \cite{Silva2004, Silva2007, Silva2007a, 
HintonLewis1978, Szwarc2002}).

The second class, that is blended Jacobi matrices (see Definition~\ref{def:3}), has been introduced in \cite{Janas2011}
as an example of unbounded Jacobi matrices having absolutely continuous spectrum equal to a finite union of compact
intervals. It has been further studied in \cite{ChristoffelI, SwiderskiTrojan2019} in the context of orthogonal
polynomials.

Before we formulate the main results of this paper, let us introduce some definitions. In our investigation, the crucial
r\^ole is played by the \emph{transfer matrix} defined as
\[
	B_j(x) = 
	\begin{pmatrix}
		0 & 1 \\
		-\frac{a_{j-1}}{a_j} & \frac{x - b_j}{a_j}
	\end{pmatrix}.
\]
We say that a sequence $(x_n : n \in \NN)$ of vectors from a normed vector space $V$
belongs to $\calD_r (V)$ for a certain $r \in \NN_0$, if it is \emph{bounded,} and for each $j \in \{1, \ldots, r\}$,
\[
	\sum_{n = 1}^\infty \big\| \Delta^j x_n \big\|^\frac{r}{j} < \infty
\]
where
\begin{align*}
	\Delta^0 x_n &= x_n, \\
	\Delta^j x_n &= \Delta^{j-1} x_{n+1} - \Delta^{j-1} x_n, \qquad j \geq 1.
\end{align*}
If $X$ is the real line with Euclidean norm we abbreviate $\calD_{r} = \calD_{r}(X)$. Given a compact set
$K \subset \CC$ and a normed vector space $R$, we denote by $\calD_{r}(K, R)$ the case when $X$ is the space of all
continuous mappings from $K$ to $R$ equipped with the supremum norm.

\begin{main_theorem} 
	\label{thm:A}
	Suppose that $A$ is a Jacobi matrix with $N$-periodically modulated entries. Let
	\[
		\calX_0(x) = \lim_{n \to \infty} X_{nN}(x)
	\]
	where 
	\[
		X_n(x) = B_{n+N-1}(x) B_{n+N-2}(x) \cdots B_n(x).
	\]
	Assume that\footnote{For a real matrix $X$ we define its discriminant as $\discr X = (\tr X)^2 - 4 \det X$.}
	$\discr \calX_0(0) > 0$. If there are a compact set $K \subset \RR$ with at least $N+1$ points,
	$r \in \NN$ and $i \in \{0, 1, \ldots, N-1 \}$, so that\footnote{By $\Mat(d, \RR)$ we denote the real matrices
	of dimension $d \times d$ with the operator norm.}
	\begin{equation}
		\label{eq:131}
		\big( X_{nN+i} : n \in \NN \big) \in \calD_r \big( K, \Mat(2, \RR) \big),
	\end{equation}
	then $A$ is self-adjoint and\footnote{For a self-adjoint operator $A$ we denote by $\sigmaEss(A), \sigmaAC(A)$
	and $\sigmaS(A)$ its essential spectrum, the essential spectrum and the singular spectrum, respectively.}
	$\sigmaEss(A) = \emptyset$.
\end{main_theorem}
Recall that a sufficient condition for self-adjointness of the operator $A$ is \emph{the Carleman's condition}
(see e.g. \cite[Corollary 6.19]{Schmudgen2017}), that is
\begin{equation}
	\label{eq:19}
    \sum_{n=0}^\infty \frac{1}{a_n} = \infty.
\end{equation}
The conclusion of Theorem~\ref{thm:A} is in strong contrast with the case when $\discr \calX_0(0) < 0$.
Indeed, if $\discr \calX_0(0) < 0$, then by \cite[Theorem A]{SwiderskiTrojan2019}, the operator $A$ is self-adjoint
if and only if the Carleman's condition is satisfied. If it is the case then $A$ is purely absolutely continuous and
$\sigma(A) = \RR$.

Under the Carleman's condition, the conclusion of Theorem~\ref{thm:A} for $r=1$ has been proven in \cite{JanasNaboko2002}
by showing that the resolvent of $A$ is compact. Furthermore, by \cite[Theorem 8]{Chihara1962}
(see also \cite[Theorem 2.6]{Szwarc2002}) it follows that if a self-adjoint Jacobi matrix $A$ is $1$-periodically
modulated with $\discr \calX_0(0) > 0$, then $\sigmaEss(A) = \emptyset$, i.e. the condition \eqref{eq:131} is
not necessary here.
\begin{main_theorem} 
	\label{thm:B}
	Suppose that $A$ is a Jacobi matrix with $N$-periodically blended entries. Let
	\[
		\calX_1(x) = \lim_{n \to \infty} X_{n(N+2)+1}(x)
	\]
	where
	\[
		X_n(x) = B_{n+N+1}(x) B_{n+N}(x) \cdots B_n(x). 
	\]
	If there are a compact set $K \subset \RR$ with at least $N+3$ points, $r \in \NN$, and $i \in \{1,2,  \ldots, N \}$,
	so that 
	\[
		\big( X_{n(N+2)+i} : n \in \NN \big) \in \calD_{r} \big( K, \Mat(2, \RR) \big),
	\]
	then $A$ is self-adjoint and 
	\[
		\sigmaS(A) \cap \Lambda = \emptyset \quad \text{and} \quad
		\sigmaAC(A) = \sigmaEss(A) = \overline{\Lambda}
	\]
	where
	\[
		\Lambda = \big\{ x \in \RR : \discr \calX_1(x) < 0 \big \}.
	\]
\end{main_theorem}
For the proof of Theorem \ref{thm:B} see Theorem \ref{thm:6}. Let us comment that in Theorem \ref{thm:B}, the absolute
continuity of $A$ follows by \cite[Theorem B]{SwiderskiTrojan2019}. Moreover, by \cite[Theorem 3.13]{ChristoffelI}
it stems that $\Lambda$ is a union of $N$ open disjoint bounded intervals. For $r = 1$ and under certain very strong
assumptions, Theorem \ref{thm:B} has been proven in \cite[Theorem 5]{Janas2011}.

The following results concerns the case when $\discr \calX_0(0) = 0$. For the proof see Theorem \ref{thm:10}.
\begin{main_theorem}
	\label{thm:C}
	Let $A$ be a Jacobi matrix with $N$-periodically modulated entries, and let $X_n$ and $\calX_0$ be defined as in
	Theorem~\ref{thm:A}. Suppose that $\calX_0(0) = \sigma \Id$ for any $\sigma \in \{-1, 1\}$, and that there are two
	$N$-periodic sequences $(s_n : n \in \NN_0)$ and $(z_n : n \in \NN_0)$, such that
	\[
		\lim_{n \to \infty} \bigg|\frac{\alpha_{n-1}}{\alpha_n} a_n - a_{n-1} - s_n\bigg| = 0,
		\qquad
		\lim_{n \to \infty} \bigg|\frac{\beta_n}{\alpha_n} a_n - b_n - z_n\bigg| = 0.
	\]
	Let $R_n = a_{n+N-1}(X_n - \sigma \Id)$. Then $(R_{kN} : k \in \NN_0)$ converges locally uniformly on $\RR$ to 
	$\calR_0$. If there are a compact set $K \subset \RR$ with at least $N+1$ points and $i \in \{1,2,  \ldots, N \}$, 
	so that
	\[
		\big( R_{nN+i} : n \in \NN \big) \in 
		\calD_{1} \big( K, \Mat(2, \RR) \big),
	\]
	then $A$ is self-adjoint and
	\[
		\sigmaS(A) \cap \Lambda = \emptyset 
		\qquad \text{and} \qquad
		\sigmaAC(A) = \sigmaEss(A) = \overline{\Lambda}
	\]
	where
	\[
		\Lambda = \big\{ x \in \RR : \discr \calR_0(x) < 0 \big\}.
	\]
\end{main_theorem}
In fact, Theorem~\ref{thm:C} completes the analysis started in \cite{PeriodicIII} where it has been shown that
$\overline{\Lambda} \subset \sigmaAC(A)$.

Finally, we investigate the case when the Carleman's condition \eqref{eq:19} is \emph{not} satisfied.
\begin{main_theorem} 
	\label{thm:D}
	Let $A$ be a Jacobi matrix with $N$-periodically modulated entries, and let $X_n$ and $\calX_0$ be defined as in
	Theorem~\ref{thm:A}. Suppose that $\calX_0(0) = \sigma \Id$ for any $\sigma \in \{-1, 1\}$, and that the Carleman's
	condition is \emph{not} satisfied. Assume that there are $i \in \{0, 1, \ldots, N-1\}$ and a sequence of positive
	numbers $(\gamma_n : n \in \NN_0)$ satisfying
	\[
		 \sum_{n=0}^\infty \frac{1}{\gamma_n} = \infty,
	\]
	such that $R_{nN+i}(0) = \gamma_n(X_{nN+i}(0) - \sigma \Id)$ converges to a non-zero matrix $\calR_i$.
	Suppose that
	\[
		\big( R_{nN+i}(0) : n \in \NN \big) \in \calD_1 \big( K, \Mat(2, \RR) \big).
	\]
	Then
	\begin{enumerate}[(i), leftmargin=2em]
		\item if $\discr \calR_i < 0$, then $A$ is \emph{not} self-adjoint;
		\item if $\discr \calR_i > 0$, then $\sigmaEss(A) = \emptyset$ provided $A$ is self-adjoint.
	\end{enumerate}
\end{main_theorem}
In fact, in Theorem \ref{thm:8a} we characterize when $A$ is self-adjoint. To illustrate Theorem \ref{thm:D},
in Section \ref{sec:KM} we consider the $N$-periodically modulated Kostyuchenko--Mirzoev's class. In this 
context we can precisely describe when the operator $A$ is self-adjoint.

In our analysis the basic objects are \emph{generalized eigenvectors} of $A$. Let us recall that
$(u_n : n \in \NN_0)$ is a generalized eigenvector associated with $x \in \CC$, if for all $n \geq 1$
\[
	\begin{pmatrix}
		u_n \\
		u_{n+1}
	\end{pmatrix} 
	=
	B_n(x) 
	\begin{pmatrix}
		u_{n-1} \\
		u_{n}
	\end{pmatrix}
\]
for a certain $(u_0, u_1) \neq (0,0)$. The spectral properties of $A$ are intimately related to the asymptotic behavior
of generalized eigenvectors. For example, $A$ is self-adjoint if and only if there is a generalized eigenvector associated
with some $x_0 \in \RR$, that is \emph{not} square-summable. In another vein, the theory of subordinacy (see 
\cite{Khan1992}) describes spectral properties of a self-adjoint $A$ in terms of asymptotic behavior of generalized
eigenvectors. In particular, it has been shown in \cite{Silva2007} that the subordinacy theory together
with some general properties of self-adjoint operators imply the following: if $K \subset \RR$ is a compact interval 
such that for each $x \in K$ there is a generalized eigenvector $(u_n(x) : n \in \NN_0)$ associated with $x \in K$, 
so that
\begin{equation} 
	\label{eq:132}
	\sum_{n=0}^\infty \sup_{x \in K} |u_n(x)|^2 < \infty,
\end{equation}
then $\sigmaEss(A) \cap K = \emptyset$. In \cite{Silva2007}, for some class of Jacobi matrices the condition~\eqref{eq:132}
has been checked with a help of uniform discrete Levinson's type theorems. In this article we take similar approach.
In particular, in Theorems~\ref{thm:2} and \ref{thm:3}, we prove our uniform Levinson's theorems. They improve the
existing results known in the literature. More precisely, Theorem~\ref{thm:2} with $r \geq 2$ in the case of negative 
discriminant, improves the pointwise theorem \cite[Theorem 3.1]{Moszynski2006}. The case of positive discriminant 
for $r > 2$ has not been studied before, even pointwise. Concerning the uniformity, Theorem~\ref{thm:2} improves \cite{Silva2004}, where for $r=1$ it was assumed that the limiting matrix is constant. Our analysis shows that this condition can be dropped (see the comment after proof of Theorem~\ref{thm:3}). We prove uniformity by constructing explicit diagonalization of the relevant matrices. The case of positive discriminant provides more technical challenges than the negative one.
If the Carleman's condition is not satisfied, our Levinson's type
theorems allowed us to study asymptotic behavior of generalized eigenvectors on the whole complex plane for a
large class of sequences $a$ and $b$. In particular, our results cover the asymptotic recently obtained by Yafaev in
\cite{Yafaev2019}, see Corollary~\ref{cor:3} for details. Let us emphasize that our approach is different than 
used in \cite{Yafaev2019}.

The organization of the paper is as follows. In Section~\ref{sec:prelim} we collect basic properties and definitions. 
In particular, we prove axillary results concerning periodically modulated and blended Jacobi matrices.
In Section~\ref{sec:stolz} we describe Stolz classes, and prove results necessary to show in Section \ref{sec:levinson}
our Levinson's type theorems which might be of independent interest. In Section~\ref{sec:essential} we apply them
to deduce Theorems~\ref{thm:A}, \ref{thm:B} and \ref{thm:C}. Finally, in Section~\ref{sec:notCarleman} we prove
Theorem~\ref{thm:D}, and study the Kostyuchenko--Mirzoev's class of Jacobi matrices in details.

\subsection*{Notation}
By $\NN$ we denote the set of positive integers and $\NN_0 = \NN \cup \{0\}$. Throughout the whole article, we write $A \lesssim B$ if there is an absolute constant $c>0$ such that
$A \le cB$. We write $A \asymp B$ if $A \lesssim B$ and $B \lesssim A$. Moreover, $c$ stands for a positive constant whose value may vary from occurrence to occurrence.

\subsection*{Acknowledgment}
The first author was partially supported by the Foundation for Polish Science (FNP) and by long term structural funding -- Methusalem grant of the Flemish Government.

\section{Preliminaries} \label{sec:prelim}
Given two sequences $a = (a_n : n \in \NN_0)$ and $b = (b_n : n \in \NN_0)$ of positive and real numbers, respectively, we define $k$th associated \emph{orthonormal} polynomials as
\[
	\begin{gathered}
		p^{[k]}_0(x) = 1, \qquad p^{[k]}_1(x) = \frac{x - b_k}{a_k}, \\
		a_{n+k-1} p^{[k]}_{n-1}(x) + b_{n+k} p^{[k]}_n(x) + a_{n+k} p^{[k]}_{n+1}(x) = 
			x p^{[k]}_n(x), \qquad n \geq 1.
	\end{gathered}
\]
We usually omit the superscript if $k = 0$. Suppose that the Jacobi matrix $A$ corresponding to the sequences $a$ and
$b$ is self-adjoint. Let us denote by $E_A$ its spectral resolution of the identity. Then for any Borel subset
$B \subset \RR$, we set
\[
	\mu(B) = \langle E_A(B) \delta_0, \delta_0 \rangle_{\ell^2(\NN_0)}
\]
where $\delta_0$ is the sequence having $1$ on the $0$th position and $0$ elsewhere. The polynomials $(p_n : n \in \NN_0)$
form an orthonormal basis of $L^2(\RR, \mu)$.

In this article, we are interested in Jacobi matrices associated to two classes of sequences that are defined in terms
of periodic Jacobi parameters. The latter are described as follows.
Let $(\alpha_n : n \in \ZZ)$ and $(\beta_n : n \in \ZZ)$ be two $N$-periodic sequences of real and positive
numbers, respectively. Let $(\mathfrak{p}_n : n \in \NN_0)$ be the corresponding polynomials, that is
\[
	\begin{gathered}
		\mathfrak{p}_0(x) = 1, \qquad \mathfrak{p}_1(x) = \frac{x-\beta_0}{\alpha_0}, \\
		\alpha_n \mathfrak{p}_{n-1}(x) + \beta_n \mathfrak{p}_n(x) 
		+ \alpha_{n} \mathfrak{p}_{n+1}(x)
		= x \mathfrak{p}_n(x), \qquad n \geq 1.
	\end{gathered}
\]
Let
\[
	\frakB_n(x) = 
	\begin{pmatrix}
		0 & 1 \\
		-\frac{\alpha_{n-1}}{\alpha_n} & \frac{x - \beta_n}{\alpha_n}
	\end{pmatrix},
	\qquad\text{and}\qquad
	\frakX_n(x) = \prod_{j = n}^{N+n-1} \mathfrak{B}_j(x), \qquad n \in \ZZ
\]
where for a sequence of square matrices $(C_n : n_0 \leq n \leq n_1)$ we have set
\[
	\prod_{k = n_0}^{n_1} C_k = C_{n_1} C_{n_1-1} \cdots C_{n_0}.
\]

\subsection{Periodic modulation}
\begin{definition}
	\label{def:2}
	We say that the Jacobi matrix $A$ associated to $(a_n : n \in \NN_0)$ and $(b_n : n \in \NN_0)$ has 
	\emph{$N$-periodically modulated entries,} if there are two $N$-periodic sequences
	$(\alpha_n : n \in \ZZ)$ and $(\beta_n : n \in \ZZ)$ of positive and real numbers, respectively, such that
	\begin{enumerate}[(a), leftmargin=2em]
	\item
	$\begin{aligned}[b]
	\lim_{n \to \infty} a_n = \infty
	\end{aligned},$
	\item
	$\begin{aligned}[b]
	\lim_{n \to \infty} \bigg| \frac{a_{n-1}}{a_n} - \frac{\alpha_{n-1}}{\alpha_n} \bigg| = 0
	\end{aligned},$
	\item
	$\begin{aligned}[b]
	\lim_{n \to \infty} \bigg| \frac{b_n}{a_n} - \frac{\beta_n}{\alpha_n} \bigg| = 0
	\end{aligned}.$
	\end{enumerate}
\end{definition}
For a Jacobi matrix $A$ with $N$-periodically modulated entries, we set
\[
	X_n = \prod_{j = n}^{N+n-1} B_j.
\]
Then for each $i \in \{0, 1, \ldots, N-1\}$ the sequence $(X_{jN+i} : j \in \NN_0)$ has a limit $\calX_i$. In view of
\cite[Proposition 3.8]{ChristoffelI}, we have $\calX_i(x) = \frakX_i(0)$ for all $x \in \CC$.

\begin{proposition} 
	\label{prop:10}
	Let $N$ be a positive integer and $\sigma \in \{-1, 1\}$. Let $A$ be a Jacobi matrix with $N$-periodically modulated
	entries so that $\frakX_0(0) = \sigma \Id$. Suppose that there are two $N$-periodic sequences $(s_n : n \in \ZZ)$
	and $(z_n : n \in \ZZ)$, such that
	\[
		\lim_{n \to \infty} \bigg|\frac{\alpha_{n-1}}{\alpha_n} a_n - a_{n-1} - s_n\bigg| = 0,
		\qquad
		\lim_{n \to \infty} \bigg|\frac{\beta_n}{\alpha_n} a_n - b_n - z_n\bigg| = 0,
	\]
	then for each $i \in \{0, 1, \ldots, N-1 \}$ the sequence 
	$\big(a_{(k+1)N+i-1} (X_{kN+i} - \sigma \Id) : k \in \NN \big)$ converges locally uniformly on $\CC$ to 
	$\calR_i$, and
	\[
		\tr \calR_i = -\sigma \lim_{k \to \infty} \big( a_{(k+1)N+i-1} - a_{kN+i-1} \big).
	\]
\end{proposition}
\begin{proof}
	According to \cite[Proposition 9]{PeriodicIII}, we have
	\begin{equation} \label{eq:102a}
		\calR_i(x) = \alpha_{i-1} \calC_i(x)+ \alpha_{i-1} \calD_i
	\end{equation}
	where 
	\[
		\calC_i(x) = x
		\begin{pmatrix}
			-\frac{\alpha_{i-1}}{\alpha_i} \Big(\frakp^{[i+1]}_{N-2}\Big)'(0) & \Big(\frakp^{[i]}_{N-1}\Big)'(0) \\
			-\frac{\alpha_{i-1}}{\alpha_i} \Big(\frakp^{[i+1]}_{N-1}\Big)'(0) & \Big(\frakp^{[i]}_{N}\Big)'(0)
		\end{pmatrix}
	\]
	and
	\begin{equation} 
		\label{eq:100}
		\calD_i = 
			\sum_{j=0}^{N-1}
			\frac{1}{\alpha_{i+j}}
			\left\{
				\prod_{m=j+1}^{N-1} \frakB_{i+m} (0)
			\right\}
			\begin{pmatrix}
				0 & 0 \\
				s_{i+j} & z_{i+j}
			\end{pmatrix}
			\left\{
				\prod_{m=0}^{j-1} \frakB_{i+m} (0)
			\right\}.
	\end{equation}
	In view of \cite[Proposition 6]{PeriodicIII}, 
	\begin{equation}
		\label{eq:102}
		\tr \calC_i \equiv 0.
	\end{equation}
	Since the trace is linear and invariant under cyclic permutations, by \eqref{eq:100} we get
	\begin{equation} 
		\label{eq:101}
		\tr \calD_i =
		\sum_{j=0}^{N-1}
			\frac{1}{\alpha_{i+j}}
			\tr \left\{
			\begin{pmatrix}
				0 & 0 \\
				s_{i+j} & z_{i+j}
			\end{pmatrix}
				\prod_{m=j+1}^{N+ j-1} \frakB_{i+m} (0)
			\right\}.
	\end{equation}
	Using \cite[Proposition 3]{PeriodicIII}
	\[
		\prod_{m=j+1}^{N+j-1} \frakB_{i+m} (0) =
		\begin{pmatrix}
			-\frac{\alpha_{i+j}}{\alpha_{i+j+1}} \frakp^{[i+j+2]}_{N-3}(0) &
			\frakp^{[i+j+1]}_{N-2}(0) \\
			-\frac{\alpha_{i+j}}{\alpha_{i+j+1}} \frakp^{[i+j+2]}_{N-2}(0) &
			\frakp^{[i+j+1]}_{N-1}(0)
		\end{pmatrix},
	\]
	thus
	\begin{align}
		\nonumber
		\tr \left\{
			\begin{pmatrix}
				0 & 0 \\
				s_{i+j} & z_{i+j}
			\end{pmatrix}
			\prod_{m=j+1}^{N+ j-1} \frakB_{i+m} (0)
		\right\}
		&=
		s_{i+j} \frakp^{[i+j+1]}_{N-2}(0) + z_{i+j} \frakp^{[i+j+1]}_{N-1}(0) \\
		\label{eq:18}
		&=
		-\sigma \frac{\alpha_{i+j}}{\alpha_{i+j-1}} s_{i+j},
	\end{align}
	where the last equality follows by \cite[formula (13)]{PeriodicIII}. Inserting \eqref{eq:18} into
	\eqref{eq:101} results in
	\[
		\tr \calD_i = -\sigma
		\sum_{j=0}^{N-1} \frac{s_{i+j}}{\alpha_{i+j-1}}.
	\]
	Hence, by \eqref{eq:102a} and \eqref{eq:102}, we get
	\[
		\tr \calR_i = \alpha_{i-1} \tr \calD_i =
		-\sigma \alpha_{i-1}
		\sum_{j=0}^{N-1} \frac{s_{i+j}}{\alpha_{i+j-1}}.
	\]
	Finally, by \cite[Proposition 3]{christoffelII}, we obtain
	\[
		\tr \calR_i
		=
		-\sigma \lim_{k \to \infty} \big( a_{(k+1)N+i-1} - a_{kN+i-1} \big),
	\]
	which completes the proof.
\end{proof}

\subsection{Periodic blend}
\begin{definition}
	The Jacobi matrix $A$ associated to $(a_n : n \in \NN_0)$ and $(b_n : n \in \NN_0)$ has 
	\emph{asymptotically $N$-periodic entries} if there are two $N$-periodic sequences
	$(\alpha_n : n \in \ZZ)$ and $(\beta_n : n \in \ZZ)$ of positive and real numbers, respectively, such that
	\begin{enumerate}[(a), leftmargin=2em]
		\item 
		$\begin{aligned}[b]
			\lim_{n \to \infty} \big|a_n - \alpha_n\big| = 0
		\end{aligned}$,
		\item
		$\begin{aligned}[b]
			\lim_{n \to \infty} \big|b_n - \beta_n\big| = 0
		\end{aligned}$.
	\end{enumerate}
\end{definition}

\begin{definition}
	\label{def:3}
	The Jacobi matrix $A$ associated with sequences $(a_n : n \in \NN_0)$ and $(b_n : n \in \NN_0)$ has 
	a \emph{$N$-periodically blended entries} if there are an asymptotically $N$-periodic Jacobi matrix $\tilde{A}$
	associated with sequences $(\tilde{a}_n : n \in \NN_0)$ and $(\tilde{b}_n : n \in \NN_0)$, and a sequence of positive
	numbers $(\tilde{c}_n : n \in \NN_0)$, such that
	\begin{enumerate}[(a), leftmargin=2em]
		\item
		$\begin{aligned}[b]
			\lim_{n \to \infty} \tilde{c}_n = \infty, \qquad\text{and}\qquad
			\lim_{m \to \infty} \frac{\tilde{c}_{2m+1}}{\tilde{c}_{2m}} = 1
		\end{aligned}$,
		\item
		$\begin{aligned}[b]
			a_{k(N+2)+i} = 
			\begin{cases}
				\tilde{a}_{kN+i} & \text{if } i \in \{0, 1, \ldots, N-1\}, \\
				\tilde{c}_{2k} & \text{if } i = N, \\
				\tilde{c}_{2k+1} & \text{if } i = N+1,
			\end{cases}
		\end{aligned}$
		\item
		$\begin{aligned}[b]
			b_{k(N+2)+i} = 
			\begin{cases}
				\tilde{b}_{kN+i} & \text{if } i \in \{0, 1, \ldots, N-1\}, \\
				0 & \text{if } i \in \{N, N+1\}.
			\end{cases}
		\end{aligned}$
	\end{enumerate}
\end{definition}
If $A$ is a Jacobi matrix having $N$-periodically blended entries, we set
\[
	X_n(x) = \prod_{j = n}^{N+n+1} B_j(x).
\]
By \cite[Proposition 3.12]{ChristoffelI}, for each $i \in \{1, 2, \ldots, N-1\}$,
\[
	\lim_{j \to \infty} B_{j(N+2)+i}(x) = \frakB_i(x),
\]
locally uniformly with respect to $x \in \CC$, thus the sequence $(X_{j(N+2)+i}:j \in \NN)$
converges to $\calX_i$ locally uniformly on $\CC$ where
\begin{equation} \label{eq:32}
	\calX_i(x) = 
	\bigg( \prod_{j = 1}^{i-1} \frakB_j(x) \bigg) \calC(x) \bigg(\prod_{j = i}^{N-1} \frakB_j(x) \bigg),
\end{equation}
and
\[
	\calC(x) = 
	\begin{pmatrix}
		0 & -1 \\
		\frac{\alpha_{N-1}}{\alpha_0} & -\frac{2x-\beta_0}{\alpha_0}
	\end{pmatrix}.
\]
Moreover, we have the following proposition.
\begin{proposition}
	\label{prop:1}
	\begin{align}
		\label{eq:21a}
		\lim_{j \to \infty} B^{-1}_{j(N+2)}(x) &= 
		\begin{pmatrix}
			0 & 0 \\
			1 & 0
		\end{pmatrix}, \\
		\label{eq:21b}
		\lim_{j \to \infty} B_{j(N+2)+N}(x) &=
		\begin{pmatrix}
			0 & 1 \\
			0 & 0
		\end{pmatrix}, \\
		\label{eq:21c}
		\lim_{j \to \infty} B_{j(N+2)+N+1}(x) &=
		\begin{pmatrix}
			0 & 1 \\
			-1 & 0
		\end{pmatrix}.
	\end{align}
	locally uniformly with respect to $x \in \CC$.
\end{proposition}
\begin{proof}
	The proposition easily follows from Definition \ref{def:3}. Indeed, we have
	\[
		B_{j(N+2)}^{-1}(x) = 
		\begin{pmatrix}
			\frac{x-\tilde{b}_{jN}}{\tilde{c}_{2j-1}} & -\frac{\tilde{a}_{jN}}{\tilde{c}_{2j-1}} \\
			1 & 0
		\end{pmatrix},
	\]
	and
	\[
		B_{j(N+2)+N}(x) 
		=
		\begin{pmatrix}
			0 & 1 \\
			-\frac{\tilde{a}_{jN+N-1}}{\tilde{c}_{2j}} & \frac{x}{\tilde{c}_{2j}}
		\end{pmatrix},
		\qquad
		B_{j(N+2)+N+1}(x)
		=
		\begin{pmatrix}
			0 & 1 \\
			-\frac{\tilde{c}_{2j}}{\tilde{c}_{2j+1}} & \frac{x}{\tilde{c}_{2j+1}}
		\end{pmatrix}.
	\]
	Thus using Definition \ref{def:3}(i) and boundedness of the sequence $(\tilde{a}_n : n \in \NN_0)$, we can compute
	the limits.
\end{proof}

\section{Stolz class}
\label{sec:stolz}
In this section we define a proper class of slowly oscillating sequences which is motivated by \cite{Stolz1994},
see also \cite[Section 2]{SwiderskiTrojan2019}.
Let $X$ be a normed space. We say that a bounded sequence $(x_n)$ belongs to $\calD_{r, s}(X)$ for certain $r \in \NN$
and $s \in \{0, 1, \ldots, r-1\}$, if for each $j \in \{1, \ldots, r-s\}$,
\[
	\sum_{n=1}^\infty \big\|\Delta^j x_n \big\|^{\frac{r}{j+s}} < \infty.
\]
Moreover, $(x_n) \in \calD_{r,s }^0(X)$, if $(x_n) \in \calD_{r,s}(X)$ and
\[
	\sum_{n=1}^\infty \|x_n\|^{\frac{r}{s}} < \infty.
\]
Let us observe that 
\[
	\calD_{r, s}(X) \subset \calD_{r, 0}(X),
	\qquad\text{and}\qquad
	\calD_{r, r-1}(X) = \calD_{1, 0}(X).
\]
To simplify the notation, if $X$ is the real line with Euclidean norm we shortly write $\calD_{r, s} 
= \calD_{r,s}(\RR)$. Given a compact set $K \subset \CC$ and a normed vector space $R$, by $\calD_{r, s}(K, R)$ we denote
the case when $X$ is the space of all continuous mappings from $K$ to $R$ equipped with the supremum norm. Moreover, given
a positive integer $N$, we say that $(x_n) \in \calD^N_{r, s}(X)$ if for each $i \in \{0, 1, \ldots, N-1 \}$,
\[
	 \big( x_{nN+i} : n \in \NN \big) \in \calD_{r,s}(X).
\]
\begin{lemma} 
	\label{lem:3}
	Let $d$ and $M$ be positive integers, and let $K_0 \subset \CC$ be a set with at least $M+1$ points. 
	Suppose that $(x_n : n \in \NN)$ is a sequence of elements from $\Mat(d, \PP_M)$ where $\PP_M$ denotes the
	linear space of complex polynomials of degree at most $M$. If there are $r \geq 1$ and $s \in \{0, 1, \ldots, r-1\}$
	so that for all $z \in K_0$,
	\[
		\big( x_n(z) : n \in \NN \big) \in \calD_{r,s} \big(\Mat(d, \CC) \big),
	\]
	then for every compact set $K \subset \CC$,
	\[
		\big( x_n : n \in \NN \big) \in \calD_{r,s} \big(K, \Mat(d, \CC) \big).
	\]
\end{lemma}
\begin{proof}
	Let $\{ z_0, z_1, \ldots, z_M \}$ be a subset of $K_0$ consisting of distinct points. By the Lagrange interpolation
	formula, we can write
	\[
		x_n(z) = \sum_{j=0}^M \ell_j(z) x_n(z_j)
	\]
	where
	\[
		\ell_j(z) = 
		\prod_{\stackrel{m = 0}{m \neq j}}^M
		\frac{z - z_m}{z_j - z_m}.
	\]
	Let $K$ be a compact subset of $\CC$. Then there is a constant $c>0$ such that for any $j \in \{0, 1, \ldots, M \}$,
	\[
		\sup_{z \in K} |\ell_j(z)| \leq c.
	\]
	Since, for each $k \geq 0$,
	\[
		\Delta^k x_n(z) = \sum_{j=0}^M \ell_j(z) \Delta^k x_n(z_j),
	\]
	we obtain
	\[
		\sup_{z \in K} \big\| \Delta^k x_n(z) \big\| 
		\leq 
		c \sum_{j=0}^M \big\| \Delta^k x_n(z_j) \big\|,
	\]
	and the conclusion follows.
\end{proof}

The following lemma is well-known and its proof is straightforward.
\begin{lemma}
	\label{lem:1}
	For any two sequences $(x_n)$ and $(y_n)$, and $j \in \NN$,
	\[
		\Delta^j(x_n y_n : n \in \NN)_n = \sum_{k = 0}^j {j \choose k} \Delta^{j-k}x_n \cdot
		\Delta^k y_{n + j - k}.
	\]
\end{lemma}

\begin{corollary}[{\cite[Corollary 1]{SwiderskiTrojan2019}}]
	\label{cor:1}
	Let $r \in \NN$ and $s \in \{0, \ldots, r-1\}$. 
	\begin{enumerate}[(i), leftmargin=2em]
		\item If $(x_n) \in \calD_{r, 0}(X)$ and $(y_n) \in \calD_{r, s}^0(X)$ then
		$(x_n y_n) \in \calD_{r, s}^0(X)$.
		\item If $(x_n), (y_n) \in \calD_{r, s}(X)$, then $(x_n y_n) \in \calD_{r, s}(X)$.
	\end{enumerate}
\end{corollary}

\begin{lemma}[{\cite[Lemma 2]{SwiderskiTrojan2019}}]
	\label{lem:2}
	Fix $r \in \NN$, $s \in \{0, \ldots, r-1\}$ and a compact set $K \subseteq \RR$. Let 
	$(f_n : n \in \NN) \in \calD_{r, s}(K, \RR)$ be a sequence of real functions on $K$ with values in $I \subseteq \RR$
	and let $F \in \calC^{r-s}(I, \RR)$. Then $(F \circ f_n : n \in \NN) \in \calD_{r, s}(K, \RR)$.
\end{lemma}

By Lemma \ref{lem:2}, we easily get the following corollary.
\begin{corollary}
	\label{cor:2}
	Let $r \in \NN$. If $(x_n) \in \calD_{r, 0}(K, \CC)$, and
	\[
		0 < \delta \leq \abs{x_n(x)},
	\]
	for all $n \in \NN$ and $x \in K$, then $(x_n^{-1} : n \in \NN) \in \calD_{r, 0}(K, \CC)$.
\end{corollary}

The next theorem is the main result of this section. 
\begin{theorem}
	\label{thm:1}
	Fix two integers $r \geq 2$ and $s \in \{0, \ldots, r-2\}$, and a compact set $K \subset \RR$. Suppose that
	$(\lambda_n^+ : n \in \NN)$ and $(\lambda_n^- : n \in \NN)$ are two uniform Cauchy sequences from 
	$\calD_{r, 0}(K, \RR)$ so that for all $x \in K$ and $n \in \NN$,
	\begin{equation}
		\label{eq:7}
		\begin{aligned}
		\lambda_n^+(x) \lambda_n^-(x) &> 0, \\
		\abs{\lambda_n^+(x)} - \abs{\lambda_n^-(x)} &\geq \delta > 0.
		\end{aligned}
	\end{equation}
	Let $(X_n : n \in \NN) \in \calD_{r, s}\big(K, \GL(2, \RR)\big)$ be such that
	\begin{equation}
		\label{eq:8}
		\sup_{x \in K} \sup_{n \in \NN} \big(\|X_n(x)\| + \|X_n^{-1}(x)\|\big) < \infty.
	\end{equation}
	Then there are sequences $(\mu_n^+ : n \in \NN), (\mu_n^- : n \in \NN) \in \calD_{r, 0}(K, \RR)$ and 
	$(Y_n : n \in \NN) \in \calD_{r, s+1}\big(K, \GL(2, \RR)\big)$ satisfying
	\begin{equation}
		\label{eq:20}
		\begin{pmatrix}
		\lambda^+_n & 0 \\
		0 & \lambda_n^-
		\end{pmatrix}
		X_n^{-1} X_{n-1} = 
		Y_n 
		\begin{pmatrix}
		\mu_n^+ & 0 \\
		0 & \mu_n^-
		\end{pmatrix}
		Y_n^{-1},
	\end{equation}
	such that $(\mu_n^+ : n \in \NN)$ and $(\mu_n^- : n \in \NN)$ are uniform Cauchy sequences with
	\begin{align*}
		\mu_n^+(x) \mu_n^-(x) &> 0, \\
		\abs{\mu_n^+(x)} - \abs{\mu_n^-(x)} &\geq \delta' > 0,
	\end{align*}
	for all $x \in K$ and $n \in \NN$. Moreover,
	\begin{equation}
		\label{eq:10}
		\lim_{n \to \infty} \sup_{x \in K} \big\| Y_n(x) - \Id \big\| = 0.
	\end{equation}
\end{theorem}
\begin{proof}
	Let
	\[
		D_n = 
		\begin{pmatrix}
		\lambda_n^+ & 0 \\
		0 & \lambda_n^-
		\end{pmatrix}.
	\]
	We set
	\[
		W_n = D_n X_n^{-1} X_{n-1} = D_n \big( \Id - X_n^{-1} \Delta X_{n-1}\big).
	\]
	By \eqref{eq:8}, we have
	\[
		\sup_{K} \big\|W_n - D_n \big\| = \sup_{K} \big\|D_n X_n^{-1} \Delta X_{n-1} \big\| 
		\leq 
		c
		\sup_K \big\| \Delta X_{n-1} \big\|.
	\]
	Since $(X_n) \in \calD_{r, s}\big(K, \GL(2, \RR)\big)$,
	\[
		\lim_{n \to \infty} \sup_K \| \Delta X_n \|= 0,
	\]
	thus
	\[
		\lim_{n \to \infty} \sup_K \big\| W_n - D_n \big\|= 0.
	\]
	In particular, $W_n$ has positive discriminant. Let $\mu^+_n$ and $\mu_n^-$ be its eigenvalues with 
	$\abs{\mu^+_n} > \abs{\mu^-_n}$. Then
	\[
		\lim_{n \to \infty} \sup_K \big|\mu_n^+ - \lambda_n^+ \big| = 0,
		\qquad\text{and}\qquad
		\lim_{n \to \infty} \sup_K \big|\mu_n^- - \lambda_n^- \big| = 0,
	\]
	and hence $(\mu^+_n : n \in \NN)$ and $(\mu_n^- : n \in \NN)$ are a uniform Cauchy sequence satisfying \eqref{eq:7}. 
	Setting
	\[
		X_n = \begin{pmatrix}
		x_{11}^{(n)} & x_{12}^{(n)} \\
		x_{21}^{(n)} & x_{22}^{(n)}
		\end{pmatrix},
		\qquad\text{and}\qquad
		W_n =
		\begin{pmatrix}
		w_{11}^{(n)} & w_{12}^{(n)} \\
		w_{21}^{(n)} & w_{22}^{(n)}
		\end{pmatrix}, 
	\]
	we obtain
	\[
		W_n
		=
		\frac{1}{\det X_n}
		\begin{pmatrix}
		\lambda_n^+ \big(x_{11}^{(n-1)} x_{22}^{(n)} - x_{21}^{(n-1)} x_{12}^{(n)}\big) &
		\lambda_n^+ \big(x_{12}^{(n-1)} x_{22}^{(n)} - x_{22}^{(n-1)} x_{12}^{(n)}\big) \\
		\lambda_n^- \big( x_{21}^{(n-1)} x_{11}^{(n)} - x_{11}^{(n-1)} x_{21}^{(n)}\big) & 
		\lambda_n^- \big( x_{22}^{(n-1)} x_{11}^{(n)} - x_{12}^{(n-1)} x_{21}^{(n)}\big)
		\end{pmatrix}.
	\]
	By \eqref{eq:8} and Corollary \ref{cor:2}, we have
	\[
		\bigg(\frac{1}{\det X_n} \bigg) \in \calD_{r, 0},
	\]
	hence by Corollary \ref{cor:1}(ii), we get that
	\[
		\big(w_{11}^{(n)} : n \in \NN \big), \big(w_{22}^{(n)} : n \in \NN \big)\in \calD_{r, 0}.
	\]
	Moreover,
	\begin{align*}
		w_{12}^{(n)} 
		&= 
		\frac{\lambda_n^+}{\det X_n} \big(x_{12}^{(n-1)} x_{22}^{(n)} - x_{22}^{(n-1)} x_{12}^{(n)}\big) \\
		&=
		\frac{\lambda_n^+}{\det X_n} \Big(\big(x_{22}^{(n)} - x_{22}^{(n-1)}\big) x_{12}^{(n)} 
		- \big(x_{12}^{(n)}-x_{12}^{(n-1)}\big) x_{22}^{(n)}\Big),
	\end{align*}
	and
	\begin{align*}
		w_{21}^{(n)} 
		&=
		\frac{\lambda_n^-}{\det X_n} \Big(\big(x_{22}^{(n)} - x_{22}^{(n-1)}\big) x_{21}^{(n)} 
		- \big(x_{21}^{(n)}-x_{21}^{(n-1)}\big) x_{22}^{(n)}\Big),
	\end{align*}
	thus, by Corollary \ref{cor:1}(i),
	\[
		\big(w_{12}^{(n)} : n \in \NN \big), \big(w_{21}^{(n)} : n \in \NN \big)\in \calD_{r, s+1}^0.
	\]
	Next, we compute the eigenvalues. We obtain
	\[
		\mu_n^+ = \frac{w_{11}^{(n)}+w_{22}^{(n)}}{2} 
		+ \frac{\sigma_n}{2} \sqrt{\discr{W_n}},\qquad\text{and}\qquad
		\mu_n^- = \frac{w_{11}^{(n)}+w_{22}^{(n)}}{2}
        - \frac{\sigma_n}{2} \sqrt{\discr{W_n}}
	\]
	where $\sigma_n = \sign{w_{11}^{(n)}}$, and
	\[
		\discr{W_n} = \big(w_{22}^{(n)} - w_{11}^{(n)}\big)^2 + 4 w_{12}^{(n)} w_{21}^{(n)}.
	\]
	Since for all $n$ sufficiently large
	\begin{equation}
		\label{eq:11}
		\big| w_{11}^{(n)}-w_{22}^{(n)} \big| \geq 
		\big|\lambda_n^+ - \lambda_n^-\big|
		- \big|w_{11}^{(n)}-\lambda_n^+ \big|
        - \big|w_{22}^{(n)}-\lambda_n^- \big| \geq \frac{\delta}{2},
	\end{equation}
	by Lemma \ref{lem:2}, we have $(\mu_n^+),(\mu_n^-) \in \calD_{r, 0}(K, \RR)$. It remains to compute the matrix 
	$Y_n$. Suppose that the equations 
	\begin{equation}
		\label{eq:30}
		W_n 
		\begin{pmatrix}
			1 \\
			v_n^+
		\end{pmatrix}
		=
		\mu_n^+ 
		\begin{pmatrix}
			1 \\
			v_n^+
		\end{pmatrix}
		\qquad\text{and}\qquad
		W_n
		\begin{pmatrix}
			v_n^- \\
			1
		\end{pmatrix}
		=
		\mu_n^-
		\begin{pmatrix}
			v_n^- \\
			1
		\end{pmatrix}
	\end{equation}
	both have solutions, then the matrix
	\[
		Y_n = 
		\begin{pmatrix}
		1 & v_n^- \\
		v_n^+ & 1
		\end{pmatrix}
	\]
	satisfies \eqref{eq:20}. Observe that equations \eqref{eq:30} are equivalent to
	\begin{align}
		\label{eq:34}
		\left\{
		\begin{aligned}
		w_{11}^{(n)} + v_n^+ w_{12}^{(n)} &= \mu_n^+, \\
		w_{21}^{(n)} + v_n^+ w_{22}^{(n)} &= \mu_n^+ v_n^+,
		\end{aligned}
		\right.
		\qquad\text{and}\qquad
		\left\{
		\begin{aligned}
		w_{11}^{(n)} v_n^- + w_{12}^{(n)} &= \mu_n^- v_n^-, \\
		w_{21}^{(n)} v_n^- + w_{22}^{(n)} &= \mu_n^-.
		\end{aligned}
		\right.
	\end{align}
	If $\sigma_n = 1$ then by \eqref{eq:11},
	\[
		w_{22}^{(n)} - w_{11}^{(n)} - \sqrt{\discr{W_n}} \leq -\frac{\delta}{2},
	\]
	otherwise
	\[
		w_{22}^{(n)} - w_{11}^{(n)} + \sqrt{\discr{W_n}} \geq \frac{\delta}{2}.
	\]
	Thus
	\[
		\big| w_{22}^{(n)} - w_{11}^{(n)} - \sigma_n \sqrt{\discr{W_n}} \big| \geq \frac{\delta}{2},
	\]
	and
	\begin{align*}
		v_n^+ 
		=
		\frac{-2 w_{21}^{(n)}}{w_{22}^{(n)} - w_{11}^{(n)}
        - \sigma_n \sqrt{\discr W_n}},
		\quad\text{and}\quad
		v_n^- 
        =
		\frac{2 w_{12}^{(n)}}{w_{22}^{(n)} - w_{11}^{(n)} 
		- \sigma_n \sqrt{\discr{W_n}}},
	\end{align*}
	satisfy the systems \eqref{eq:34}. In view of \eqref{eq:11}, Corollary \ref{cor:2} and Corollary \ref{cor:1}(i),
	we conclude that $(v_n^+), (v_n^-) \in \calD_{r, s+1}^0(K, \RR)$. Finally, Lemma \ref{lem:2} implies that $(Y_n)$
	belongs to $\calD_{r, s+1}\big(K, \GL(2, \RR)\big)$. Because
	\[
		\lim_{n \to \infty} \sup_K{\abs{v_n^+}} = \lim_{n \to \infty} \sup_K{\abs{v_n^-}} = 0,
	\]
	we easily obtain \eqref{eq:10}. 
\end{proof}

\begin{corollary}
	The sequences $(\mu^-_n)$ and $(\mu^+_n)$ converge to the same limit as $(\lambda^-_n)$ and $(\lambda^+_n)$, 
	respectively.
\end{corollary}

\section{Levinson's type theorems} \label{sec:levinson}
In this section we develop discrete variants of the Levinson's theorem. There are two cases we need to distinguish
according to whether the limiting matrix has two different eigenvalues or not.

\subsection{Different eigenvalues}
\begin{theorem}
	\label{thm:2}
	Let $(X_n : n \in \NN)$ be a sequence of continuous mappings defined on $\RR$ with values in $\GL(2, \RR)$ 
	that converges uniformly on a compact set $K$ to the mapping $\calX$ with $\discr \calX(x) \neq 0$ and
	$\det \calX(x) > 0$ for each $x \in K$. If $\discr \calX > 0$, we additionally assume that for all $x \in K$,
	\begin{equation}
		\label{eq:35}
		\big|[\calX(x)]_{1, 1} - \lambda_1(x)\big| > 0
		\quad\text{and}\quad
		\big|[\calX(x)]_{2, 2} - \lambda_2(x)\big| > 0
	\end{equation}
	where $\lambda_1$ and $\lambda_2$ are continuous functions  on $K$ so that $\lambda_1(x)$ and $\lambda_2(x)$ are
	eigenvalues of $\calX(x)$. Let $(E_n : n \in \NN)$ be a sequence of continuous mappings defined on $\RR$ with values
	in $\Mat(2, \CC)$ such that
	\begin{equation}
		\label{eq:40}
		\sum_{n = 1}^\infty \sup_K{\| E_n\|} < \infty.
	\end{equation}
	If $(X_n : n \in \NN)$ belongs to $\calD_{r, 0}\big(K, \GL(2, \RR) \big)$ for a certain $r \geq 1$ and
	$\eta$ is a continuous eigenvalue of $\calX$, then there are continuous mappings
	$\Phi_n: K \rightarrow \CC^2$, $\mu_n: K \rightarrow \CC$, and $v : K \rightarrow \CC^2$, satisfying
	\[
		\Phi_{n+1} = (X_n + E_n)\Phi_n
	\]
	and
	\[
		\lim_{n \to \infty} \sup_{x \in K}{|\mu_n(x) - \eta(x)|} = 0,
	\]
	such that
	\begin{equation}
		\label{eq:14}
		\lim_{n \to \infty}
		\sup_{x \in K}{
		\bigg\|
		\frac{\Phi_n(x)}{\prod_{j=1}^{n-1} \mu_j(x)} - v(x)
		\bigg\|}
		=
		0
	\end{equation}
	whereas $v(x)$ is an eigenvector of $\calX(x)$ corresponding to $\eta(x)$ for each $x \in K$.
\end{theorem}
\begin{proof}
	Suppose that $\discr \calX(x) > 0$ and $\det \calX(x) > 0$ for all $x \in K$. In particular, $\tr \calX(x) \neq 0$
	for all $x \in K$. Let $\lambda^+$ and $\lambda^-$ denote the eigenvalues of $\calX$ such that 
	$|\lambda^+| > |\lambda^-|$, namely we set
	\[
		\lambda^+(x) = \frac{\tr \calX(x) + \sigma \sqrt{\discr \calX(x)}}{2},
		\qquad\text{and}\qquad
		\lambda^-(x) = \frac{\tr \calX(x) - \sigma \sqrt{\discr \calX(x)}}{2}
	\]
	where $\sigma = \sign{\tr \calX}$. Without loss of generality we can assume that \eqref{eq:35} is satisfied with
	$\lambda_1 = \lambda^+$ and $\lambda_2 = \lambda^-$, since otherwise we consider mappings conjugated by
	\[
		J = 
		\begin{pmatrix}
			0 & 1\\
			1 & 0
		\end{pmatrix}.
	\]
	Select $\delta > 0$ such that for all $x \in K$,
	\[
		\big|[\calX(x)]_{1, 1} - \lambda^+(x)\big| \geq 2 \delta
		\quad\text{and}\quad
		\big|[\calX(x)]_{2, 2} - \lambda^-(x)\big| \geq 2 \delta,
	\]
	and
	\[
		\discr \calX(x) \geq 2 \delta^2, \qquad \det \calX(x) \geq 2 \delta^2.
	\]
	Since $(\discr X_n : n \in \NN)$ converges uniformly on $K$, there is $M \geq 1$ such that for all $n \geq M$ and 
	$x \in K$,
	\begin{equation}
		\label{eq:43}
		\discr X_n(x) \geq \delta^2 \qquad
		\det X_n(x) \geq \delta^2.
	\end{equation}
	Hence, the matrix $X_n(x)$ has two eigenvalues
	\[
		\lambda_n^+(x) = \frac{\tr X_n(x) + \sigma \sqrt{\discr X_n(x)}}{2},
		\qquad\text{and}\qquad
		\lambda_n^-(x) = \frac{\tr X_n(x) - \sigma \sqrt{\discr X_n(x)}}{2},
	\]
	By increasing $M$, we can also assume that for all $n \geq M$ and $x \in K$,
	\begin{equation}
		\label{eq:41}
		\big|[X_n(x)]_{1, 1} - \lambda^+_n(x)\big| \geq \delta
		\quad\text{and}\quad
		\big|[X_n(x)]_{2, 2} - \lambda^-_n(x)\big| \geq \delta.
	\end{equation}
	Then setting
	\[
		C_{n, 0} =
		\begin{pmatrix}
			\frac{[X_n]_{1,2}}{\lambda^+_n-[X_n]_{1,1}} & 1 \\
			1 & \frac{[X_n]_{2,1}}{\lambda^-_n - [X_n]_{2,2}}
		\end{pmatrix}
		\qquad\text{and}\qquad
		D_{n, 0} =
		\begin{pmatrix}
			\lambda_n^+ & 0 \\
			0 & \lambda_n^-
		\end{pmatrix},
	\]
	we obtain
	\[
		X_n = C_{n, 0} D_{n, 0} C_{n, 0}^{-1}.
	\]
	In view of \eqref{eq:43} and \eqref{eq:41}, by Corollaries \ref{cor:2} and \ref{cor:1}, the sequences
	$(C_{n, 0} : n \geq M)$ and $(D_{n, 0} : n \geq M)$ belong to $\calD_{r, 0}\big(K, \GL(2, \RR)\big)$.
	If $r \geq 2$, in view of \eqref{eq:43} we can apply Theorem \ref{thm:1} to get two sequences of mappings
	\[
		(C_{n,1} : n \geq M) \in \calD_{r, 1}\big(K, \GL(2, \RR) \big),
		\qquad\text{and}\qquad
		(D_{n, 1} : n \geq M) \in \calD_{r, 0}\big(K, \GL(2, \RR) \big),
	\]
	such that
	\[
		D_{n, 0} C_{n, 0}^{-1} C_{n-1, 0} = C_{n, 1} D_{n, 1} C_{n, 1}^{-1},
	\]
	and
	\[
		D_{n, 1} = 
		\begin{pmatrix}
			\gamma_{n, 1}^+ & 0 \\
			0 & \gamma_{n, 1}^-
		\end{pmatrix}.
	\]
	Then for $n \geq M+1$,
	\begin{align*}
		X_{n+1} X_n
		&=
		(C_{n+1,0} D_{n+1, 0} C_{n+1, 0}^{-1}) (C_{n, 0} D_{n, 0} C_{n, 0}^{-1}) \\
		&=
		C_{n+1, 0} (D_{n+1, 0} C_{n+1, 0}^{-1} C_{n, 0}) (D_{n, 0} C_{n, 0}^{-1} C_{n-1, 0}) C_{n-1, 0}^{-1} \\
		&=
		C_{n+1, 0} (C_{n+1, 1} D_{n+1, 1} C_{n+1,1}^{-1}) (C_{n, 1} D_{n, 1} C_{n, 1}^{-1}) C_{n-1, 0}^{-1} \\
		&=
		C_{n+1, 0} C_{n+1, 1} (D_{n+1, 1} C_{n+1, 1}^{-1} C_{n, 1}) (D_{n, 1} C_{n, 1}^{-1} C_{n-1, 1}) 
		(C_{n-1, 0} C_{n-1, 1})^{-1}.
	\end{align*}
	By repeated application of Theorem \ref{thm:1} for $k \in \{2, 3, \ldots, r-1\}$, we can find sequences
	\begin{equation}
		\label{eq:6}
		(C_{j, k} : j \geq M+k) \in \calD_{r, k}\big(K, \GL(2, \RR)\big),
		\quad\text{and}\quad
		(D_{j, k} : j \geq M+k) \in \calD_{r, 0}\big(K, \GL(2, \RR)\big),
	\end{equation}
	such that
	\[
		D_{j, k-1} C_{j, k-1}^{-1} C_{j-1, k-1} = C_{j, k} D_{j, k} C_{j, k}^{-1}.
	\]
	Hence,
	\[
		X_{n+1} X_n = 
		Q_{n+1} \big(D_{n+1, r-1} C_{n+1, r-1}^{-1} C_{n, r-1}\big) \big(D_{n, r-1} C_{n, r-1}^{-1} C_{n-1, r-1} \big)
		Q_{n-1}^{-1}
	\]
	where
	\[
		Q_m = C_{m, 0} C_{m, 1} \cdots C_{m, r-1}.
	\]
	Notice that 
	\[
		\lim_{m \to \infty} Q_m = 
		\begin{pmatrix}
			\frac{[\calX]_{1,2}}{\lambda^+-[\calX]_{1,1}} & 1 \\
			1 & \frac{[\calX]_{2,1}}{\lambda^--[\calX]_{2,2}}
		\end{pmatrix}
	\]
	uniformly on $K$.

	Let us now consider the recurrence equation
	\begin{align*}
		\Psi_{k+1} &= Q_{2k+1}^{-1} (X_{2k+1} + E_{2k+1})(X_{2k} + E_{2k}) Q_{2k-1} \Psi_k \\
		&=\big(Y_k + R_k + F_k\big) \Psi_k, \qquad k \geq M
	\end{align*}
	where
	\[
		Y_k = D_{2k+1, r-1} D_{2k, r-1}
		=
		\begin{pmatrix}
			\gamma_{2k+1, r-1}^+ \gamma_{2k, r-1}^+ & 0 \\
			0 & \gamma_{2k+1, r-1}^- \gamma_{2k, r-1}^-
		\end{pmatrix},
	\]
	\[
		R_k = (D_{2k+1, r-1} C_{2k+1, r-1}^{-1} C_{2k, r-1})
		(D_{2k,r-1} C_{2k, r-1}^{-1} C_{2k-1, r-1}) - D_{2k+1, r-1} D_{2k, r-1},
	\]
	and
	\[
		F_k = Q_{2k+1}^{-1} X_{2k+1} E_{2k} Q_{2k-1} + Q_{2k+1}^{-1} E_{2k+1} X_{2k} Q_{2k-1} 
		+ Q_{2k+1}^{-1} E_{2k+1} E_{2k} Q_{2k-1}.
	\]
	Since
	\begin{align*}
		R_k &= -D_{2k+1, r-1} C_{2k+1, r-1}^{-1} \big( \Delta C_{2k, r-1}\big) D_{2k, r-1} \\
		&\phantom{=}- D_{2k+1, r-1} C_{2k+1, r-1}^{-1} C_{2k, r-1} D_{2k, r-1} C_{2k, r-1}^{-1}
		\big( \Delta C_{2k-1, r-1} \big),
	\end{align*}
	we easily see that
	\[
		\|R_k\| \leq
		c
		\big(\big\|\Delta C_{2k, r-1}\big\| + \big\|\Delta C_{2k-1, r-1} \big\| \big),
	\]
	which together with \eqref{eq:6} implies that
	\[
		\sum_{k = M}^\infty \sup_K \|R_k\| < \infty.
	\]
	In view of \eqref{eq:40} we also get
	\[
		\sum_{k = M}^{\infty} \sup_K \|F_k\| < \infty.
	\]
	Let us consider the case $\eta = \lambda^-$. The sequence $(\gamma_{n, r-1}^- : n \geq M)$ converges to $\lambda^-$,
	thus there are $n_0 \geq M$ and $\delta' > 0$, so that for all $n \geq n_0$,
	\[
		\bigg|\frac{\gamma^+_{n, r-1}}{\gamma^-_{n, r-1}}\bigg| \geq 1 + \frac{\delta}{\abs{\gamma^-_{n, r-1}}}
		\geq 1 + \delta',
	\]
	thus for all $k_1 \geq k_0 \geq n_0$,
	\[
		\prod_{j = k_0}^{k_1} 
		\bigg|\frac{\gamma^+_{2j+1, r-1}}{\gamma^-_{2j+1, r-1}}\bigg| \cdot 
		\bigg|\frac{\gamma^+_{2j, r-1}}{\gamma^-_{2j, r-1}} \bigg|
		\geq (1+ \delta')^{2(k_1-k_0)}.
	\]
	In particular, $(Y_k : k \geq n_0)$ satisfies the uniform Levinson's condition, see
	\cite[Definition 2.1]{Silva2004}. Therefore, in view of \cite[Theorem 4.1]{Silva2004}, there is a sequence
	$(\Psi_k : k \geq n_0)$ such that
	\[
		\lim_{n \to \infty}
		\sup_{x \in K}{
		\bigg\|
		\frac{\Psi_k(x)}{\prod_{j = n_0}^{k-1} \gamma^-_{2j+1, r-1}(x) \gamma^-_{2j, r-1}(x)}
		-
		e_2
		\bigg\|} = 0
	\]
	where
	\[
		e_2 = \begin{pmatrix}
		0 \\
		1
		\end{pmatrix}.
	\]
	In fact, in the proof of \cite[Theorem 4.1]{Silva2004} the author used supremum norm with respect to the
	parameter, thus when all the mappings in \cite[Theorem 4.1]{Silva2004} are continuous (or holomorphic) with respect
	to this parameter, the functions $\Phi_k$ are continuous (or holomorphic, respectively).

	We are now in the position to define $(\Phi_n : n \geq 2 n_0)$. Namely, for $x \in K$ and $n \geq 2n_0$, we set
	\[
		\Phi_{n}(x) = 
		\begin{cases}
			Q_{2k-1}(x) \Psi_k(x) & \text{if } n = 2k, \\
			(X_{2k}(x) + E_{2k}(x)) Q_{2k-1}(x) \Psi_k(x) & \text{if } n = 2k+1.
		\end{cases}
	\]
	As it is easy to check, $(\Phi_n : n \geq n_0)$ satisfies
	\[
		\Phi_{n+1} = (X_n + E_n) \Phi_n.
	\]
	Observe that for
	\[
		v^- = 
		\begin{pmatrix}
			1 \\
			\frac{[\calX]_{2,1}}{\lambda^--[\calX]_{2,2}}
		\end{pmatrix}
	\]
	we obtain
	\[
		\lim_{k \to \infty} \sup_{x \in K}{ \big\|Q_{2k-1}(x) e_2 - v^-(x)\big\|} = 0,
	\]
	and
	\[
		\lim_{k \to \infty} \sup_{x \in K}{
		\bigg\| \frac{X_{2k}(x) + E_{2k}(x)}{\gamma^-_{2k, r-1}(x)} Q_{2k-1}(x) e_2 - v^-(x) \bigg\|} = 0.
	\]
	Therefore, \eqref{eq:14} is satisfied for $(\mu_n : n \in \NN)$ defined on $K$ by the formula
	\[
		\mu_n = 
		\begin{cases}
			1 & \text{for } n < n_0, \\
			\gamma^-_{n, r-1} &\text{for } n \geq n_0.
		\end{cases}
	\]
	This completes the proof. The reasoning when $\eta = \lambda^+$ is analogous.

	If $\discr \calX(x) < 0$ for $x \in K$, the argument is simpler. First, let us observe that the matrix $\calX(x)$ has
	real entries, thus $|[\calX(x)]_{1, 2}| > 0$ for all $x \in K$. Since $(X_n : n \in \NN)$ converges uniformly on $K$,
	there are $\delta > 0$ and $M \geq 1$, such that for all $n \geq M$ and $x \in K$,
	\[
		\discr X_n(x) \leq -\delta,\qquad\text{and}\qquad |[X_n(x)]_{1, 2}| > \delta.
	\]
	Therefore, for each $x \in K$, the matrix $X_n(x)$ has two eigenvalues $\lambda_n$ and $\overline{\lambda_n}$ where
	\[
		\lambda_n(x) = \frac{\tr X_n(X) + i\sqrt{\abs{\discr X_n(x)}}}{2}.
	\]
	Hence, setting
	\[
		C_{n, 0} =
		\begin{pmatrix}
			1 & 1 \\
			\frac{\lambda_n - [X_n]_{1,1}}{[X_n]_{1,2}} & \frac{\overline{\lambda_n} - [X_n]_{1,1}}{[X_n]_{1,2}}
		\end{pmatrix},
		\qquad\text{and}\qquad
		D_{n, 0} =
		\begin{pmatrix}
			\lambda_n & 0 \\
			0 & \overline{\lambda_n}
		\end{pmatrix},
	\]
	we obtain
	\[
		X_n = C_{n, 0} D_{n, 0} C_{n, 0}^{-1}.
	\]
	Moreover, $(C_{n, 0} : n \geq M)$ and $(D_{n, 0} : n \geq M)$ belong to $\calD_{r, 0}\big(K, \GL(2, \CC) \big)$.
	If $r \geq 2$, then by \cite[Theorem 1]{SwiderskiTrojan2019}, there are two sequences of matrices
	\[
		(C_{n, 1} : n \geq M) \in \calD_{r, 1}\big(K, \GL(2, \CC) \big), \qquad\text{and}\qquad
		(C_{n, 1} : n \geq M) \in \calD_{r, 0}\big(K, \GL(2, \CC) \big),
	\]
	such that
	\[
		D_{n, 0} C_{n, 0}^{-1} C_{n, 0} = C_{n, 1} D_{n, 1} C_{n, 1}^{-1},
	\]
	and
	\[
		D_{n, 1} = 
		\begin{pmatrix}
			\gamma_{n, 1} & 0 \\
			0 & \overline{\gamma_{n, 1}}
		\end{pmatrix}.
	\]
	By repeated application of \cite[Theorem 1]{SwiderskiTrojan2019}, for $k \in \{2, 3, \ldots, r-1\}$, we can
	find sequences
	\[
		(C_{j, k} : j \geq M+k) \in \calD_{r, k}\big(K, \GL(2, \CC)\big),
		\qquad\text{and}\qquad
		(D_{j, k} : j \geq M+k) \in \calD_{r, 0}\big(K, \GL(2, \CC)\big),
	\]
	such that
	\[
		D_{j, k-1} C_{j, k-1}^{-1} C_{j-1, k-1} = C_{j, k} D_{j, k} C_{j, k}^{-1}.
	\]
	Hence,
	\[
		X_{n+1}X_n = Q_{n+1} \big(D_{n+1, r-1} C_{n+1, r-1}^{-1} C_{n, r-1} \big) 
		\big(D_{n, r-1} C_{n, r-1}^{-1} C_{n-1, r-1}\big) Q_{n-1}^{-1},
	\]
	where
	\[
		Q_m = C_{m, 0} C_{m, 1} \cdots C_{m, r-1}.
	\]
	Notice that 
	\[
		\lim_{m \to \infty} Q_m = 
		\begin{pmatrix}
			1 & 1 \\
		    \frac{\lambda - [\calX]_{1,1}}{[\calX]_{1,2}} & \frac{\overline{\lambda} - [\calX]_{1,1}}{[\calX]_{1,2}}
		\end{pmatrix}
	\]
	uniformly on $K$. 

	We next consider the recurrence equation
	\begin{align*}
		\Psi_{k+1}
		&= Q_{2k+1}^{-1} (X_{2k+1} + E_{2k+1})(X_{2k} + E_{2k})Q_{2k-1} \Psi_k \\
		&= (Y_k + R_k + F_k) \Psi_k, \qquad k \geq M
	\end{align*}
	where 
	\[
		Y_{k} = D_{2k+1, r-1} D_{2k, r-1} =
		\begin{pmatrix}
			\gamma_{2k+1, r} \gamma_{2k, r} & 0 \\
			0 & \overline{\gamma_{2k+1, r}} \overline{\gamma_{2k, r}}
		\end{pmatrix},
	\]
	\[
		R_k = (D_{2k+1, r-1} C_{2k+1, r-1}^{-1} C_{2k, r-1})
		(D_{2k,r-1} C_{2k, r-1}^{-1} C_{2k-1, r-1})-D_{2k+1, r-1} D_{2k, r-1},
	\]
	and
	\[
		F_k = Q_{2k+1}^{-1} X_{2k+1} E_{2k} Q_{2k-1} + Q_{2k+1}^{-1} E_{2k+1} X_{2k} Q_{2k-1}
		+Q_{2k+1}^{-1} E_{2k+1} E_{2k} Q_{2k-1}.
	\]
	Suppose that $\eta = \overline{\lambda}$. Since for all $n \geq M$,
	\[
		\bigg|\frac{\gamma_{n, r-1}}{\overline{\gamma_{n, r-1}}} \bigg| = 1,
	\]
	the sequence $(Y_k : k \geq M)$ satisfies the uniform Levinson's condition. Therefore, by
	\cite[Theorem 4.1]{Silva2004}, there is a sequence $(\Psi_k : k \geq M)$ such that
	\[
		\lim_{k \to \infty}\sup_{x \in K}{ \bigg\|\frac{\Psi_k(x)}
		{\prod_{j = M}^{k-1} \gamma_{2j+1, r-1}(x) \overline{\gamma_{2j, r-1}(x)}} - e_2\bigg\|} = 0.
	\]
	Hence, $(\Phi_n : n \geq 2M)$ defined by the formula 
	\[
		\Phi_{n}(x) = 
		\begin{cases}
			Q_{2k-1}(x) \Psi_k(x) & \text{if } n = 2k, \\
			(X_{2k}(x) + E_{2k}(x)) Q_{2k-1}(x) \Psi_k(x) & \text{if } n = 2k+1.
		\end{cases}
	\]
	together with
	\[
		\mu_n = 
		\begin{cases}
			1 & \text{for } n < 2 M, \\
			\gamma_{n, r-1} & \text{for } n \geq 2 M,
		\end{cases}
	\]
	satisfies \eqref{eq:14}. This completes the proof of the theorem.
\end{proof}

The following lemma guarantees that in the case of positive discriminant Theorem \ref{thm:2} can at least be applied
locally.
\begin{lemma}
	\label{lem:4}
	Suppose that $X$ is a continuous mapping defined on a closed interval $I \subset \RR$ with values in $\Mat(2, \RR)$
	that has positive discriminant on $I$. Let $\lambda_1, \lambda_2: I \rightarrow \RR$, be continuous functions
	so that $\lambda_1(x)$ and $\lambda_2(x)$ are the distinct eigenvalues of $X(x)$. Then for each 
	$x \in I$ there is an open interval $I_x$ containing $x$ such that
	\begin{enumerate}[(i), leftmargin=2em]
		\item
		for all $y \in I_x \cap I$,
		\[
			\big([X(y)]_{1,1} - \lambda_1(y)\big)\big([X(y)]_{2,2} - \lambda_2(y)\big) \neq 0,
		\]
	\end{enumerate}
	or
	\begin{enumerate}[(i), resume, leftmargin=2em]
		\item
		for all $y \in I_x \cap I$,
		\[
			\big([X(y)]_{1,1} - \lambda_2(y)\big)\big([X(y)]_{2,2} - \lambda_1(y)\big) \neq 0.
		\]
	\end{enumerate}
\end{lemma}
\begin{proof}
	Let $x \in I$. Since $\discr X(x) > 0$, we have $\lambda_1(x) \neq \lambda_2(x)$. 
	By the continuity of $X$, it is enough to show that
	\[
		\big([X(x)]_{1,1} - \lambda_1(x)\big)\big([X(x)]_{2,2} - \lambda_2(x)\big) \neq 0,
	\]
	or
	\[
		\big([X(x)]_{1,1} - \lambda_2(x)\big)\big([X(x)]_{2,2} - \lambda_1(x)\big) \neq 0.
	\]
	If neither of the conditions is met, we would have
	\[
		[X(x)]_{1,1} = \lambda_1(x) \quad\text{and}\quad [X(x)]_{2,2} = \lambda_1(x),
	\]
	or
	\[
		[X(x)]_{1,1} = \lambda_2(x) \quad\text{and}\quad [X(x)]_{2,2} = \lambda_2(x).
	\]
	Thus $\tr X(x)$ equals $2 \lambda_1(x)$ or $2 \lambda_2(x)$, but neither of the situations is possible. 
\end{proof}
The following corollary gives a Levinson's type theorem in the case when the limit $\calX$ is a constant
matrix. It may be proved in much the same way as Theorem \ref{thm:2}. Here, the condition \eqref{eq:35} can be dropped
since $\calX$ is a constant matrix.
\begin{corollary}
	\label{cor:4}
	Let $(X_n : n \in \NN)$ be a sequence of matrices belonging to $\calD_1 \big(\GL(2, \RR)\big)$ 
	convergent to the matrix $\calX$. Let $(E_n : n \in \NN)$ be a sequence of continuous (or holomorphic) 
	mappings defined on a compact set $K \subset \CC$ with values in $\Mat(2, \CC)$, such that 
	\[
		\sum_{n = 1}^\infty \sup_K \|E_n\| < \infty.
	\]
	Suppose that $\discr \calX \neq 0$ and $\det \calX > 0$. If $\eta$ is an eigenvalue of $\calX$, then there
	are continuous (or holomorphic, respectively) mappings $\Phi_n: K \rightarrow \CC^2$, satisfying
	\[
		\Phi_{n+1} = (X_n + E_n) \Phi_n
	\]
	such that
	\[
		\lim_{n \to \infty}
		\sup_{x \in K}{\bigg\|
		\frac{\Phi_n(x)}{\prod_{j = 1}^{n-1} \mu_j} - v
		\bigg\|}=0
	\]
	where $v$ is the eigenvector of $\calX$ corresponding to $\eta$, and $\mu_n$ is the eigenvalue of $X_n$ such
	that
	\[
		\lim_{n \to \infty} |\mu_n - \eta| = 0.
	\]
\end{corollary}

\subsection{Perturbations of the identity}

\begin{theorem}
	\label{thm:3}
	Let $(X_n : n \in \NN)$ be a sequence of continuous mappings defined on $\RR$
	with values in $\GL(2, \RR)$ that uniformly converges on a compact interval $K$ to $\sigma \Id$ for a certain
	$\sigma \in \{-1, 1\}$. Suppose that there is a sequence of positive numbers $(\gamma_n : n \in \NN_0)$
	such that $R_n = \gamma_n(X_n - \sigma \Id)$ converges uniformly on $K$ to the mapping $\calR$ satisfying
	$\discr \calR(x) \neq 0$ for all $x \in K$. If $\discr \calR > 0$, we additionally assume that
	\begin{equation}
		\label{eq:110}
		\sum_{n=0}^\infty \frac{1}{\gamma_n} = \infty.
	\end{equation}
	Let $(E_n : n \in \NN)$ be a sequence of continuous mappings defined on $\RR$ with values in $\Mat(2, \CC)$, 
	such that
	\begin{equation}
		\label{eq:37}
		\sum_{n = 1}^\infty \sup_K{\|E_n\|} < \infty.
	\end{equation}
	If $(R_n : n \in \NN)$ belongs to $\calD_{1, 0}\big(K, \Mat(2, \RR)\big)$ and $\eta$ is a continuous eigenvalue of
	$\calR$, then there are $n_0 \geq 1$ and continuous mappings $\Phi_n : K \rightarrow \CC^2$, 
	$\mu_n : K \rightarrow \CC$, and $v : K \rightarrow \CC^2$ satisfying
	\[
		\Phi_{n+1} = (X_n  + E_n) \Phi_n,
	\]
	such that
	\begin{equation}
		\label{eq:17}
		\lim_{n \to \infty}
		\sup_{x \in K}{\bigg\|\frac{\Phi_n(x)}{\prod_{j = n_0}^{n-1} 
		\big( \sigma + \gamma_j^{-1} \mu_j(x)\big)} - v(x) \bigg\|} = 0
	\end{equation}
	where for each $x \in K$, $v(x)$ is an eigenvector of $\calR(x)$ corresponding to $\eta(x)$, and $\mu_n(x)$
	is an eigenvalue of $R_n(x)$ such that
	\[
		\lim_{n \to \infty} \sup_{x \in K}{\big|\mu_n(x) - \eta(x)\big|} = 0.
	\]
\end{theorem}
\begin{proof}
	Let us first consider the case of positive discriminant. There is $\delta > 0$ such that for all $x \in K$,
	\[
		\discr \calR(x) \geq 2\delta^2.
	\]
	Then the matrix $\calR(x)$ has two eigenvalues
	\[
		\xi^+(x) = \frac{\tr \calR (x) + \sigma \sqrt{\discr \calR(x)}}{2},
		\qquad\text{and}\qquad
		\xi^-(x) = \frac{\tr \calR (x) - \sigma \sqrt{\discr \calR(x)}}{2}.
	\]
	Since $(R_n : n \in \NN)$ converges uniformly on $K$, there is $M \geq 1$, such that for all $n \geq M$ and
	$x \in K$,
	\begin{equation} 
		\label{eq:108}
		\discr R_n(x) \geq \delta^2,
		\qquad\text{and}\qquad
		\gamma_n \geq \delta.
	\end{equation}
	In particular, the matrix $R_{n}$ has two eigenvalues
	\begin{equation}
		\label{eq:107}
		\xi^+_n(x) = \frac{\tr R_n(x) + \sigma \sqrt{\discr R_n(x)}}{2}, 
		\qquad\text{and}\qquad
		\xi^-_n(x) = \frac{\tr R_n(x) - \sigma \sqrt{\discr R_n(x)}}{2}.
	\end{equation}
	Now, let us consider the collection of intervals $\{I_x : x \in K\}$ determined in Lemma \ref{lem:4} for 
	the mapping $\calR$. By compactness of $K$ we can find a finite subcollection $\{I_1, \ldots, I_J\}$ that covers $K$.
	Let us consider the case $\eta = \xi^-$. It is clear that
	\[
		\lim_{n \to \infty} \sup_{x \in K}{\big\|\xi^-_n(x) - \eta(x)\big\|} = 0.
	\]
	Suppose that on each $K_j = \overline{I_j} \cap K$, one can find $\Phi_n^{(j)}$ and $v^{(j)}$ so that
	\begin{equation}
		\label{eq:46}
		\lim_{n \to \infty} \sup_{x \in K_j}{
		\bigg\|
		\frac{\Phi^{(j)}_n(x)}{\prod_{m = n_0}^{n-1} \big(\sigma + \gamma_m^{-1} \mu_m(x)\big)}
		-
		v^{(j)}(x)
		\bigg\|
		}
		=0.
	\end{equation}
	Let $\{\psi_1, \ldots, \psi_J\}$ be the continuous partition of unity subordinate to the covering
	$\{I_1, \ldots , I_J\}$, that is $\psi_j$ is a continuous non-negative function with $\supp \psi_j \subset I_j$, so
	that
	\[
		\sum_{j = 1}^J \psi_j \equiv 1.
	\]
	We set
	\[
		\Phi_n = \sum_{j = 1}^J \Phi^{(j)}_n \psi_j, \qquad\text{and}\qquad
		v = \sum_{j = 1}^J v^{(j)} \psi_j.
	\]
	Observe that $v(x)$ is an eigenvector of $\calR(x)$ corresponding to $\eta(x)$ for all $x \in K$. Moreover,
	since $\psi_j$ is supported inside $I_j$,
	\begin{align*}
		\lim_{n \to \infty}
		\sup_{x \in K}{
		\bigg\|
		\frac{\Phi_n(x)}{\prod_{m = n_0}^{n-1} \big(\sigma + \gamma_m^{-1} \mu_m(x)\big)}
		-
		v(x)
		\bigg\|} 
		\leq
		\lim_{n \to \infty}
		\sum_{j = 1}^J
		\sup_{x \in K_j}{
		\bigg\|
		\frac{\Phi_n^{(j)}(x)}{\prod_{m = n_0}^{n-1} \big(\sigma + \gamma_m^{-1} \mu_m(x)\big)}
		-
		v^{(j)}(x)
		\bigg\|}
		=
		0.
	\end{align*}
	Therefore, it is sufficient to prove \eqref{eq:46} for $K = K_j$ where $j \in \{ 1, \ldots, J \}$.
	To simplify the notation, we drop the dependence on $j$. Without loss of generality, we can assume that 
	for each $x \in K$,
	\[
		\big|[\calR(x)]_{1,1} - \xi^+(x)\big| \geq 2\delta,
		\qquad\text{and}\qquad
		\big|[\calR(x)]_{2,2} - \xi^-(x) \big| \geq 2\delta.
	\]
	Since $(R_n : n \in \NN)$ converges to $\calR$ uniformly on $K$, there are $M \geq 1$ such that for all
	$x \in K$ and $n \geq M$,
	\begin{equation}
		\label{eq:49}
		\big| [R_n(x)]_{1,1} - \xi^+_n(x) \big| \geq \delta,
		\qquad\text{and}\qquad
		\big| [R_n(x)]_{2,2} - \xi^-_n(x) \big| \geq \delta.
	\end{equation}
	Now, we can define
	\[
		C_n =
		\begin{pmatrix}
			\frac{[R_n]_{1,2}}{\xi^+_n - [R_n]_{1,1}} & 1 \\
			1 & \frac{[R_n]_{2,1}}{\xi^-_n-[R_n]_{2,2}}
		\end{pmatrix},
		\qquad\text{and}\qquad
		\tilde{D}_n(x) =
		\begin{pmatrix}
			\xi^+_n(x) & 0 \\
			0 & \xi^-_n(x)
		\end{pmatrix}.
	\]
	Then
	\[
		R_{n}(x) = C_n(x) \tilde{D}_n(x) C_n^{-1}(x),
	\]
	and in view of \eqref{eq:49}, \eqref{eq:108}, \eqref{eq:107}, Corollary~\ref{cor:1} and Lemma~\ref{lem:2}, 
	we conclude that
	\begin{equation}
		\label{eq:113}
		(C_n : n \geq M) \in \calD_{1,0} \big( K, \GL(2, \RR) \big).
	\end{equation}
	Notice that
	\begin{equation}
		\label{eq:16}
		\lim_{n \to \infty} C_n = 
		\begin{pmatrix}
			\frac{[\calR]_{1,2}}{\xi^+ - [\calR]_{1,1}} & 1 \\
			1 & \frac{[\calR]_{2,1}}{\xi^- - [\calR]_{2,2}}		
		\end{pmatrix}
	\end{equation}
	uniformly on $K$. Since
	\[
		X_{n} = \sigma \Id + \frac{1}{\gamma_n} R_{n},
	\]	
	we obtain
	\[
		X_{n}(x) = C_n(x) D_n(x) C_n^{-1}(x)
	\]
	where
	\[
		D_n = \sigma \Id + \frac{1}{\gamma_n} \tilde{D}_n.
	\]
	Hence, eigenvalues of $X_n$ are
	\begin{equation}
		\label{eq:52}
		\lambda^+_n = \sigma + \frac{1}{\gamma_n} \xi^+_n, \qquad\text{and}\qquad
		\lambda^-_n = \sigma + \frac{1}{\gamma_n} \xi^-_n.
	\end{equation}
	Let us now consider the recurrence equation
	\begin{align*}
		\Psi_{n+1} &=
		C_{n+1}^{-1} (X_n + E_n)C_n \Psi_n\\
		&=(D_n + F_n) \Psi_n 
	\end{align*}
	where
	\[
		F_n = -C^{-1}_{n+1} \big( \Delta C_n \big) D_n + C_{n+1}^{-1} E_n C_n.
	\]
	By \eqref{eq:16}, we easily see that
	\[
		\| F_n \| \leq c \big(\| \Delta C_n \| + \|E_n\|\big),
	\]
	which together with \eqref{eq:113} and \eqref{eq:37} gives
	\[
		\sum_{n=M}^\infty \sup_K \| F_n \| < \infty.
	\]
	Next, in view of \eqref{eq:52}, \eqref{eq:107} and \eqref{eq:108}, for $n \geq M$,
	\begin{align} \label{eq:126}
		\bigg| \frac{\lambda^+_n}{\lambda^-_n} \bigg|
		= 
		\bigg| 1 + \frac{1}{\gamma_n} \frac{\sqrt{\discr R_n}}{1 + \frac{\sigma}{\gamma_n} \xi_n^-}\bigg| 
		\geq
		1 + \frac{\sqrt{\discr R_n}}{2 \gamma_n}
		\geq
		\exp\bigg(\frac{\delta}{4 \gamma_n} \bigg),
	\end{align}
	after possibly enlarging $M$. Therefore, for all $n_2 > n_1 \geq M$,
	\[
		\prod_{n=n_1}^{n_2} 
		\bigg| \frac{\lambda^+_n}{\lambda^-_n} \bigg|
		\geq 
		\exp\bigg(\frac{\delta}{4} \sum_{n = n_1}^{n_2} \frac{1}{\gamma_n} \bigg).
	\]
	Hence, \eqref{eq:110} guarantees that the sequence $(D_n : n \geq M)$ satisfies the uniform Levinson's
	condition. Let us remind that we are considering $\eta = \xi^-$. In view of \cite[Theorem 4.1]{Silva2004}, there
	is a sequence $(\Psi_n : n \geq M)$ such that
	\[
		\lim_{n \to \infty} \sup_{x \in K}{
		\bigg\|
		\frac{\Psi_n(x)}{\prod_{j=M}^{n-1} \lambda_{j}^-(x)} - e_2
		\bigg\|}
		=
		0.
	\]
	Now, for $x \in K$ and $n \geq M$, we set
	\[
		\Phi_n(x) = C_n(x) \Psi_n(x).
	\]
	It is easy to verify that $(\Phi_n : n \geq M)$ satisfies 
	\[
		\Phi_{n+1} = (X_n + E_n) \Phi_n.
	\]
	Setting
	\[
		v = 
		\begin{pmatrix}
			1 \\
			\frac{[\calR]_{2,1}}{\xi^-- [\calR]_{2,2}}
		\end{pmatrix}
	\]
	by \eqref{eq:16}, we get 
	\[
		\lim_{n \to \infty} \sup_{x \in K}{\big\| C_n(x) e_2 - v(x) \big\|} = 0,
	\]
	which completes the proof of \eqref{eq:17} for $K = K_j$, and the case of positive discriminant follows.
	
	When $\discr \calR < 0$ on $K$, the reasoning is similar. Since the matrix $\calR$
	has real entries, $[\calR(x)]_{1,2} \neq 0$ for all $x \in K$. Therefore, for $n \geq M$, we can set
	\[
		C_n =
		\begin{pmatrix}
			1 & 1 \\
			\frac{\xi^+_n - [R_n]_{1,1}}{[R_n]_{1, 2}} & \frac{\xi^-_n - [R_n]_{1,1}}{[R_n]_{1, 2}}
		\end{pmatrix}
	\]
	where
	\[
		\xi^+_n(x) = \frac{\tr R_n(x) + i \sqrt{|\discr R_n(x)|}}{2}, 
		\qquad\text{and}\qquad
		\xi^-_n(x) = \frac{\tr R_n(x) - i \sqrt{|\discr R_n(x)}|}{2}.
	\]
	Since
	\[
		\bigg| \frac{\lambda^+_n}{\lambda^-_n} \bigg| = 1,
	\]
	the sequence $(D_n : n \geq M)$ satisfies the uniform Levinson's condition. The rest of the proof runs as before.
\end{proof}

The method of the proof used in Theorem \ref{thm:3}, can be also applied in the case of different eigenvalues
and $r = 1$. In particular, the condition \eqref{eq:35} can be dropped.

The proof of the following corollary is analogous to the proof of Theorem \ref{thm:3}.
\begin{corollary}
	\label{cor:5}
	Let $(X_n : n \in \NN)$ be a sequence of matrices in $\GL(2, \RR)$ convergent to the matrix $\sigma \Id$
	for a certain $\sigma \in \{-1, 1\}$. Suppose that there is a sequence of positive numbers $(\gamma_n : n \in \NN_0)$ 
	such that $R_n = \gamma_n(X_n - \sigma \Id)$ converges to the matrix $\calR$ satisfying $\discr \calR \neq 0$.
	If $\discr \calR > 0$, we additionally assume
	\[
		\sum_{n = 0}^\infty \frac{1}{\gamma_n} = \infty.
	\]
	Let $(E_n : n \in \NN)$ be is a sequence of continuous (or holomorphic) mappings on a compact set $K \subset \CC$ with values in
	$\Mat(2, \CC)$, such that
	\[
		\sum_{n = 1}^\infty \sup_K \|E_n\| < \infty.
	\]
	If $(R_n: n \in \NN)$ belongs to $\calD_{1, 0}(\Mat(2, \RR))$, and $\eta$ is an eigenvalue of $\calR$, then there
	are $n_0 \geq 1$ and continuous (or holomorphic, respectively) mappings $\Phi_n: K \rightarrow \CC^2$, satisfying
	\[
		\Phi_{n+1} = (X_n + E_n) \Phi_n,
	\]
	and such that
	\[
		\lim_{n \to \infty}
		\sup_{x \in K}{
		\bigg\|
		\frac{\Phi_n(x)}{\prod_{j = n_0}^{n-1} (\sigma + \gamma_j^{-1} \mu_j)} - v
		\bigg\|}=0
	\]
	where $v$ is an eigenvector of $\calR$ corresponding to $\eta$, $\mu_n$ is the eigenvalue of $R_n$ such that
	\[
		\lim_{n \to \infty} |\mu_n - \eta| = 0.
	\]
\end{corollary}

In the following proposition we describe a way to estimate the denominator in \eqref{eq:17}.
\begin{proposition} \label{prop:7}
	Let $(X_n : n \in \NN)$ be a sequence of mappings defined on $\RR$ with values in $\GL(2, \RR)$ convergent 
	on a compact set $K$ to $\sigma \Id$ for a certain $\sigma \in \{-1, 1\}$. Suppose that there is a sequence of
	positive numbers $(\gamma_n : n \in \NN)$ satisfying
	\[
		\lim_{n \to \infty} \Gamma_n = \infty \qquad \text{where} \qquad
		\Gamma_n = \sum_{j=1}^n \frac{1}{\gamma_j},
	\]
	such that $R_n = \gamma_n (X_n - \sigma \Id)$ converges uniformly on $K$ to the mapping $\calR$. Assume that
	$\discr \calR(x) > 0$ for all $x \in K$, and
	\begin{equation}
		\label{eq:50}
		\sum_{n=1}^\infty \Gamma_n \cdot \sup_{x \in K} \| R_{n+1}(x) - R_n(x) \| < \infty.
	\end{equation}
	Then there is $n_0$ such that for all $n \geq n_0$, and $x \in K$,
	\begin{equation}
		\label{eq:38}
		\prod_{j = n_0}^n \big|\sigma + \gamma_j^{-1} \mu_j^-(x) \big|^2 \asymp
		\exp \Big(\Gamma_n \big( \sigma \tr \calR(x) - \sqrt{\discr \calR(x)} \big) \Big)
	\end{equation}
	and
	\begin{equation}
		\label{eq:39}
		\prod_{j = n_0}^n \big|\sigma + \gamma_j^{-1} \mu_j^+(x) \big|^2 \asymp
		\exp \Big(\Gamma_n \big( \sigma \tr \calR(x) + \sqrt{\discr \calR(x)} \big) \Big)
	\end{equation}
	where
	\begin{equation}
		\label{eq:56}
		\mu_j^- = \frac{1}{2} \Big(\tr R_{j} - \sigma \sqrt{\discr R_{j}}\Big),
		\qquad\text{and}\qquad
		\mu_j^+ = \frac{1}{2}\Big(\tr R_{j} + \sigma \sqrt{\discr R_{j}}\Big).
	\end{equation}
	The implicit constants in \eqref{eq:38} and \eqref{eq:39} are independent of $x$ and $n$.
\end{proposition}
\begin{proof}
	Since $\discr \calR > 0$ on $K$, there is $n_0$ such that for all $j \geq n_0$ and $x \in K$,
	$\discr R_{j}(x) > 0$. Thus $R_{j}$ has two eigenvalues given by the formulas \eqref{eq:56}.
	By possible enlarging $n_0$, for all $n \geq n_0$, we have
	\[
		\log\Big(
		\prod_{j = n_0}^n \big|\sigma + \gamma_j^{-1} \mu_j^- \big|^2
		\Big)
		\asymp
		\sum_{j = n_0}^n \frac{1}{\gamma_j} \Big(\sigma \tr R_{j} - \sqrt{\discr R_{j}}\Big)
	\]
	uniformly on $K$. Let 
	\[
		A_n^- = \sigma \tr R_{n} - \sqrt{\discr R_{n}},
		\qquad
		A_{\infty}^- = \sigma \tr \calR - \sqrt{\discr \calR}.
	\]
	Since for $m \geq n$,
	\begin{align*}
		\big|A_n^- - A_m^- \big| \cdot \Gamma_n
		&\leq
		c \sum_{k = n}^\infty \big\|R_{k+1} - R_{k} \big\| \cdot \Gamma_n \\
		&\leq
		c \sum_{k = n}^\infty \big\|R_{k+1} - R_{k} \big\| \cdot \Gamma_k,
	\end{align*}
	we obtain
	\begin{equation}
		\label{eq:59}
		\sup_K{\big| A_n^- - A_{\infty}^- \big|} \cdot \Gamma_n
		\leq
		c.
	\end{equation}
	Now, by the summation by parts, we get
	\begin{align*}
		\sum_{j = n_0}^n \frac{1}{\gamma_j} A_j^-
		&=
		(\Gamma_n - \Gamma_{n_0-1}) A_{\infty}^- + \sum_{j = n_0}^n (\Gamma_j - \Gamma_{j-1}) 
		(A_j^- - A_{\infty}^-) \\
		&=
		\Gamma_n A^-_{\infty} - \Gamma_{n_0-1} A_{n_0}^- + \Gamma_n (A_n^- - A_{\infty}^-) +
		\sum_{j = n_0}^{n-1} \Gamma_j (A_j^- - A_{j+1}^-),
	\end{align*}
	thus, by \eqref{eq:50} and \eqref{eq:59},
	\[
		\sup_K{\bigg| \sum_{j = n_0}^n \frac{1}{\gamma_j} A_j^- - A_{\infty}^- \cdot \Gamma_n\bigg|}
		\leq c.
	\]
	Hence,
	\[
		\prod_{j = n_0}^n \big|\sigma + \gamma_j^{-1} \mu_j^- \big|^2 
		\asymp
		\exp \Big(\Gamma_n \big( \sigma \tr \calR - \sqrt{\discr \calR} \big) \Big),
	\]
	uniformly on $K$. The proof of \eqref{eq:39} is similar.
\end{proof}

\section{Essential spectrum for positive discriminant} 
\label{sec:essential}
In this section we prove the main results of the paper.
\begin{theorem}
	\label{thm:6}
	Let $N$ and $r$ be positive integers and $i \in \{1, 2, \ldots, N\}$. Let $A$ be a Jacobi matrix with
	$N$-periodically blended entries. If there is a compact set $K_0 \subset \RR$ with at least $N+3$ points so
	that
	\begin{equation} 
		\label{eq:127}
		\big( X_{n(N+2)+i} : n \in \NN \big) \in \calD_{r,0} \big( K_0, \Mat(2, \RR) \big),
	\end{equation}
	then $A$ is self-adjoint and 
	\[
		\sigmaS(A) \cap \Lambda = \emptyset \quad \text{and} \quad
		\sigmaAC(A) = \sigmaEss(A) = \overline{\Lambda}
	\]
	where
	\[
		\Lambda = \big\{ x \in \RR : \discr \calX_1(x) < 0 \big\}
	\]
	wherein $\calX_1$ is given by the formula \eqref{eq:32}.
\end{theorem}
\begin{proof}
	Fix $x_0 \in \RR \setminus \overline{\Lambda}$. Let $I$ be an open interval containing $x_0$ such that
	$\overline{I} \subset \RR \setminus \overline{\Lambda}$. Since $\discr \calX_1 = \discr \calX_i$,
	we have $\discr \calX_i > 0$ on $\overline{I}$. Thus the matrix $\calX_i$ has two different eigenvalues
	$\lambda^+$ and $\lambda^-$. Since $\det \calX_i \equiv 1$, we can select them in such a way that
	\[
		\abs{\lambda^-} < 1 < \abs{\lambda^+}.
	\]
	Let $I_0$ be an open interval determined by Lemma \ref{lem:4} for $x_0$ and the mapping $\calX_i: \overline{I}
	\rightarrow \Mat(2, \RR)$. Without loss of generality we can assume that, for all $x \in I_0$,
	\[
		\big| [\calX_i(x)]_{1, 1} - \lambda^-(x)\big| > 0,
		\qquad\text{and}\qquad
		\big| [\calX_i(x)]_{2, 2} - \lambda^+(x)\big| > 0.
	\]
	Let $K = \overline{I_0}$. In view of Lemma \ref{lem:3},
	\begin{equation} 
		\label{eq:130}
		\big( X_{j(N+2)+i} : j \in \NN \big) \in \calD_{r,0} \big( K, \GL(2, \RR) \big).
	\end{equation}
	Now, by Theorem \ref{thm:2}, there are sequences $(\Phi^\pm_n : n \geq n_0)$ and $(\mu_n^\pm : n \in \NN)$,
	such that
	\begin{equation} \label{eq:138}
		\lim_{n \to \infty} \sup_{x \in K}{
		\bigg\| \frac{\Phi^\pm_n(x)}{\prod_{j=1}^n \mu_j^\pm(x)} - v^\pm(x) \bigg\|} = 0.
	\end{equation}
	where $v^\pm$ is a continuous eigenvector of $\calX_i$ corresponding to $\lambda^\pm$. We set
	\[
		\phi_1^\pm = B_1^{-1} \cdots B_{n_0(N+2)+i-1}^{-1} \Phi^\pm_{n_0},
	\]
	and for $n \geq 1$,
	\begin{equation}
		\label{eq:24}
		\phi^\pm_{n+1} = B_n \phi^\pm_n.
	\end{equation}
	Then for $k(N+2)+i' > n_0(N+2) + i$ with $i' \in \{0, 1, \ldots, N+1\}$, we have
	\begin{equation} \label{eq:139}
		\phi^\pm_{k(N+2)+i'}
		=
		\begin{cases}
			B^{-1}_{k(N+2)+i'} B^{-1}_{k(N+2)+i'+1} \cdots B^{-1}_{k(N+2)+i-1} \Phi^\pm_k
			&\text{if } i' \in \{0, 1, \ldots, i-1\}, \\
			\Phi^\pm_k 
			&\text{if } i' = i, \\
			B_{k(N+2)+i'-1} B_{k(N+2)+i'-2} \cdots B_{k(N+2)+i} \Phi_k^\pm
			&\text{if } i' \in \{i+1, \ldots, N+1\}.
		\end{cases}
	\end{equation}
	Since for $i' \in \{1, \ldots, i-1\}$,
	\[
		\lim_{k \to \infty}
		B^{-1}_{k(N+2)+i'} B^{-1}_{k(N+2)+i'+1} \cdots B^{-1}_{k(N+2)+i-1}
		=
		\frakB^{-1}_{i'} \frakB^{-1}_{i'-1} \cdots \frakB^{-1}_{i-1},
	\]
	we obtain
	\begin{equation}
		\label{eq:25}
		\lim_{k \to \infty}
		\sup_K{
		\bigg\|
		\frac{\phi^{-}_{k(N+2)+i'}}{\prod_{j=1}^{k-1} \mu_j^-} 
		- 
		\frakB^{-1}_{i'} \frakB^{-1}_{i'-1} \cdots \frakB^{-1}_{i-1} v^-
		\bigg\|
		}
		=0.
	\end{equation}
	Analogously, for $i' \in \{i+1, \ldots, N\}$, we get
	\begin{equation}
		\label{eq:26}
		\lim_{k \to \infty}
		\sup_K{
		\bigg\|
		\frac{\phi^{-}_{k(N+2)+i'}}{\prod_{j=1}^{k-1} \mu_j^-}
		-
		\frakB_{i'-1} \frakB_{i'-2} \cdots \frakB_i v^-
		\bigg\|
		}
		=0.
	\end{equation}
	Lastly, by Proposition \ref{prop:1},
	\begin{equation}
		\label{eq:22}
		\lim_{k \to \infty}
		\sup_K
		\Bigg\|
		\frac{\phi^{-}_{k(N+2)}}{\prod_{j = 0}^{k-1} \mu_j^-}
		-
		\begin{pmatrix}
			0 & 0 \\
			1 & 0
		\end{pmatrix}
		\frakB_{1}^{-1} \frakB_{2}^{-1} \cdots \frakB_{i-1}^{-1} v^-
		\Bigg\|
		=0
	\end{equation}
	and	
	\begin{equation}
		\label{eq:23}
		\lim_{k \to \infty}
		\sup_K
		\Bigg\|
		\frac{\phi^\pm_{k(N+2)+N+1}}{\prod_{j=1}^{k-1} \mu_j^-}
		-
		\begin{pmatrix}
			0 & 1 \\
			0 & 0
		\end{pmatrix}
		\frakB_{N-1} \frakB_{N-2} \cdots \frakB_i v^-
		\Bigg\|
		=0.
	\end{equation}
	Since $(\phi^\pm_n : n \in \NN)$ satisfies \eqref{eq:24}, the sequence $(u^\pm_n(x) : n \in \NN_0)$ 
	defined as
	\[
		u^\pm_n(x) = 
		\begin{cases}
			\langle \phi^\pm_1(x), e_1 \rangle & \text{if } n = 0, \\
			\langle \phi^\pm_n(x), e_2 \rangle & \text{if } n \geq 1,
		\end{cases}
	\]
	is a generalized eigenvector associated to $x \in K$. Observe that $(u_0, u_1) \neq 0$ on $K$. 
	Indeed, otherwise there is $x \in K$, so that $\phi^\pm_1(x) = 0$, hence $\phi^\pm_n(x) = 0$ for 
	all $n \in \NN$. Therefore, $v^\pm(x) = 0$, which is impossible. Now, in view of \eqref{eq:138} and \eqref{eq:139},
	there are constants $c > 0$ and $\delta > 0$ such that for all $n \geq n_0$ and $x \in K$,
	\[
		\big| u^+_{n(N+2)+i-1}(x) \big|^2 + \big| u^+_{n(N+2)+i}(x) \big|^2 = 
		\big\| \phi^+_{n(N+2)+i}(x) \big\|^2 \geq c
		\prod_{j=n_0}^{n-1} |\mu^+_j(x)|^2 \geq c (1 + \delta)^n.
	\]
	Moreover, by \eqref{eq:25}--\eqref{eq:23}, for all $n \geq n_0$, $i' \in \{0, 1, \ldots, N+1 \}$, and $x \in K$,
	\[
		\big\| \phi^-_{n(N+2)+i'}(x) \big\|^2 \leq c \prod_{j=n_0}^{n-1} |\mu^-_j(x)|^2 \leq c (1+\delta)^{-n}.
	\]
	Consequently, for any $x \in K$,
	\[
		\sum_{n = 0}^\infty \abs{u^+_n(x)}^2 = \infty
	\]
	which shows that $A$ is self-adjoint (see \cite[Theorem 6.16]{Schmudgen2017}). Moreover,
	\[
		\sum_{n = 0}^\infty \sup_{x \in K}{\abs{u^-_n(x)}}^2 < \infty,
	\]
	thus by the proof of \cite[Theorem 5.3]{Silva2007},
	\[
		\sigmaEss(A) \cap K = \emptyset.
	\]
	Therefore, for all $x_0 \in \RR \setminus \overline{\Lambda}$ there is an open interval $I_0$ containing $x_0$
	such that $\sigmaEss(A) \cap I_0 = \emptyset$. Consequently, $\sigmaEss(A) \subseteq \overline{\Lambda}$.
	In view of \eqref{eq:130}, \cite[Theorem B]{SwiderskiTrojan2019} implies that $A$ is purely absolutely continuous
	on $\Lambda$, and $\overline{\Lambda} \subset \sigmaAC(A)$. This completes the proof.
\end{proof}

\begin{remark}
	The proof of \cite[Corollary 6.7]{Swiderski2020} entails that \eqref{eq:127} is satisfied for any compact
	set $K \subset \RR$, and all $i \in \{ 1, 2, \ldots, N \}$, provided that
	\[
		\bigg( \frac{1}{a_n} : n \in \NN \bigg), 
		\big( b_n : n \in \NN \big) \in \calD^{N+2}_{r,0} \quad \text{and} \quad
		\bigg( \frac{a_{n(N+2)+N}}{a_{n(N+2)+N+1}} : n \in \NN \bigg) \in \calD_{r,0}.
	\]
\end{remark}

Essentially the same reasoning as in the proof of Theorem \ref{thm:6} leads to the following theorem.
\begin{theorem}
	\label{thm:7}
	Let $N$ and $r$ be positive integers and $i \in \{0, 1, \ldots, N-1\}$. Let $A$ be a Jacobi matrix with 
	$N$-periodically modulated entries. If $|\tr \frakX_0(0)| > 2$ and there is a compact set $K_0 \subset \RR$ with
	at least $N+1$ points so that 
	\begin{equation} 
		\label{eq:128}
		\big( X_{nN+i} : n \in \NN \big) \in \calD_{r,0} \big( K_0, \Mat(2, \RR) \big),
	\end{equation}
	then $A$ is self-adjoint and $\sigmaEss(A) = \emptyset$.
\end{theorem}

\begin{remark}
	The proof of \cite[Corollary 8]{SwiderskiTrojan2019} implies that \eqref{eq:128} is satisfied for any compact 
	set $K \subset \RR$, and all $i \in \{0, 1, \ldots, N-1 \}$, provided that
	\[
		\bigg( \frac{a_{n-1}}{a_n} : n \in \NN \bigg), 
		\bigg( \frac{b_n}{a_n} : n \in \NN \bigg),
		\bigg( \frac{1}{a_n} : n \in \NN \bigg) \in \calD^N_{r,0}.
	\]
\end{remark}

We next consider the case when $\frakX_0$ has equal eigenvalues.
\begin{theorem} 
	\label{thm:10}
	Let $N$ and $r$ be positive integers and $i \in \{0, 1, \ldots, N-1\}$. Let $A$ be a Jacobi matrix with 
	$N$-periodically modulated entries so that $\frakX_0(0) = \sigma \Id$ for a certain $\sigma \in \{-1, 1\}$. Suppose 
	that there are two $N$-periodic sequences $(s_n : n \in \NN_0)$ and $(z_n : n \in \NN_0)$, such that
	\[
		\lim_{n \to \infty} \bigg|\frac{\alpha_{n-1}}{\alpha_n} a_n - a_{n-1} - s_n\bigg| = 0,
		\qquad
		\lim_{n \to \infty} \bigg|\frac{\beta_n}{\alpha_n} a_n - b_n - z_n\bigg| = 0.
	\]
	Let
	\[ 
		R_n = a_{n+N-1} \big(X_n - \sigma \Id\big).
	\]
	Then $(R_{nN} : n \in \NN)$ converges to $\calR_0$ locally uniformly on $\RR$.
	If there is a compact set $K_0 \subset \RR$ with at least $N+1$ points such that
	\begin{equation} \label{eq:129}
		\big( R_{nN+i} : n \in \NN \big) \in 
		\calD_{1,0} \big( K_0, \Mat(2, \RR) \big),
	\end{equation}
	then $A$ is self-adjoint and 
	\[
		\sigmaS(A) \cap \Lambda = \emptyset \qquad \text{and} \qquad
		\sigmaAC(A) = \sigmaEss(A) = \overline{\Lambda}
	\]
	where
	\[
		\Lambda = \big\{ x \in \RR : \discr \calR_0(x) < 0 \big\}.
	\]
\end{theorem}
\begin{proof}
	In view of Proposition~\ref{prop:10}, there is $c > 0$ such that for all $k \geq 0$,
	\begin{align*}
		a_{kN+i} 
		&= 
		\sum_{j=0}^{k-1} \big( a_{(j+1)N+i} - a_{jN+i} \big) + a_i \\
		&\leq
		c(k+1).
	\end{align*}
	Therefore,
	\[
		\sum_{n = 0}^{k_0N+i} \frac{1}{a_n} 
		\geq 
		\sum_{k = 0}^{k_0} \frac{1}{a_{kN+i}}
		\geq 
		\frac{1}{c} \sum_{k=1}^{k_0} \frac{1}{k}.
	\]
	Thus, the Carleman's condition is satisfied, and so $A$ is self-adjoint.

	Thanks to Lemma~\ref{lem:3}, for any compact set $K \subset \RR$,
	\[
		\big( X_{jN+i} : j \in \NN_0 \big),
		\big( R_{nN+i} : n \in \NN \big) \in 
		\calD_{1,0} \big( K, \Mat(2, \RR) \big).
	\]
	Since $\discr \calR_0 = \discr \calR_i$, by \cite[Criterion 5.8]{Moszynski2009} together with
	\cite[Proposition 5.7]{Moszynski2009} and \cite[Theorem 5.6]{Moszynski2009}, we conclude that $A$ is purely
	absolutely continuous on $\Lambda$ and $\overline{\Lambda} \subset \sigmaAC(A)$. Hence, it remains to show that
	$\sigmaEss(A) \subset\overline{\Lambda}$. To do so, we fix a compact set $K \subset \RR \setminus \overline{\Lambda}$
	with non-empty interior. Since $\discr \calR_i > 0$ on $K$, for each $x \in K$ the matrix $\calR_i(x)$ has two
	distinct eigenvalues
	\[
		\xi^+(x) = \frac{\tr \calR_{i}(x) + \sigma \sqrt{\discr \calR_i(x)}}{2},
		\qquad\text{and}\qquad
		\xi^-(x) = \frac{\tr \calR_i(x) - \sigma \sqrt{\discr \calR_i(x)}}{2}.
	\]
	Moreover, $(\discr R_{jN+i} : j \in \NN)$ converges uniformly on $K$, thus there are $M \geq 1$ and $\delta > 0$
	such that for all $j \geq M$ and $x \in K$,
	\[
		\discr R_{jN+i}(x) \geq \delta.
	\]
	Therefore, $R_{jN+i}(x)$ has two distinct eigenvalues
	\[
		\xi^+_j(x) = \frac{\tr R_{jN+i}(x) + \sigma \sqrt{\discr R_{jN+i}(x)}}{2},
		\qquad\text{and}\qquad
		\xi^-_j(x) = \frac{\tr R_{jN+i}(x) - \sigma \sqrt{\discr R_{jN+i}(x)}}{2}.
	\]
	Since
	\[
		X_n = \sigma \Id + \frac{1}{a_{n+N-1}} R_n,
	\]
	the eigenvalues of $X_{jN+i}(x)$ are
	\[
		\lambda_j^+(x) = \sigma + \frac{\xi_j^+(x)}{a_{(j+1)N+i-1}},
		\qquad\text{and}\qquad
		\lambda_j^-(x) = \sigma + \frac{\xi_j^-(x)}{a_{(j+1)N+i-1}}.
	\]
	By Theorem \ref{thm:3}, there is a sequence $(\Phi_n : n \geq n_0)$ such that
	\begin{equation}
		\label{eq:42}
		\lim_{n \to \infty}
		\sup_{x \in K} \bigg\|\frac{\Phi_n(x)}{\prod_{j = n_0}^{n-1} \lambda_j^-(x)} - v^-(x) \bigg\|= 0
	\end{equation}
	where $v^-$ is a continuous eigenvector of $\calR_i$ corresponding to $\xi^-$. We set
	\[
		\phi_1 = B_1^{-1} \cdots B^{-1}_{n_0} \Phi_{n_0},
	\]
	and for $n \geq 1$,
	\begin{equation}
		\label{eq:15}
		\phi_{n+1} = B_n \phi_n.
	\end{equation}
	Then for $kN+i' > n_0N+i$ with $i' \in \{0, 1, \ldots, N-1\}$, we have
	\[
		\phi_{kN+i'} = 
		\begin{cases}
			B_{kN+i'}^{-1} B_{kN+i'+1}^{-1} \cdots B_{kN+i-1}^{-1} \Phi_k 
			&\text{if } i' \in \{0, 1, \ldots, i-1\}, \\
			\Phi_k & \text{if } i ' = i,\\\
			B_{kN+i'-1} B_{kN+i'-2} \cdots B_{kN+i} \Phi_k &
			\text{if } i' \in \{i+1, \ldots, N-1\}.
		\end{cases}
	\]
	Since for $i' \in \{0, 1, \ldots, i-1\}$,
	\[
		\lim_{k \to \infty} B_{kN+i'}^{-1} B_{kN+i'+1}^{-1} \cdots B_{kN+i-1}^{-1}
		=
		\frakB_{i'}^{-1}(0) \frakB_{i'+1}^{-1}(0) \cdots \frakB_{i-1}^{-1}(0),
	\]
	we obtain
	\begin{equation}
		\label{eq:27}
		\lim_{k \to \infty}
		\sup_K{
		\bigg\|
		\frac{\phi_{kN+i'}}{\prod_{j = n_0}^{k-1} \lambda^-_j} - 
		\frakB_{i'}^{-1}(0) \frakB_{i'+1}^{-1}(0) \cdots \frakB_{i-1}^{-1}(0)
		v^-
		\bigg\|
		} = 0.
	\end{equation}
	Analogously, for $i' \in \{i+1, \ldots, N-1\}$,
	\begin{equation}
		\label{eq:28}
		\lim_{k \to \infty}
		\sup_K{
		\bigg\|
		\frac{\phi_{kN+i'}}{\prod_{j = n_0}^{k-1} \lambda^-_j} -
		\frakB_{i'-1}(0) \frakB_{i'-2}(0) \cdots \frakB_{i}(0)
		v^-
		\bigg\|}
		=0.
	\end{equation}
	Since $(\phi_n : n \in \NN)$ satisfies \eqref{eq:15}, the sequence $(u_n(x) : n \in \NN_0)$ defined as
	\[
		u_n(x) = \begin{cases}
			\langle \phi_1(x), e_1 \rangle & \text{if } n = 0, \\
			\langle \phi_n(x), e_2 \rangle & \text{if } n \geq 1,
			\end{cases}
	\]
	is a generalized eigenvector associated to $x \in K$. By \eqref{eq:42}, \eqref{eq:27} and \eqref{eq:28}, for each
	$i' \in \{0, 1, \ldots, N-1\}$, $n > n_0$, and $x \in K$,
	\begin{equation}
		\label{eq:29}
		|u_{nN+i'}(x)|
		\leq
		c 
		\prod_{j = n_0}^{n-1} |\lambda_j^-(x)|.
	\end{equation}
	Since $(R_{nN+i} : n \in \NN)$ converges to $\calR_i$ uniformly on $K$, and
	\[
		\lim_{n \to \infty} a_n = \infty,
	\]
	there is $M \geq n_0$, such that for $n \geq M$,
	\[
		\frac{|\tr R_{nN+i}(x)| + \sqrt{\discr R_{nN+i}(x)}}{ a_{(n+1)N+i-1}} \leq 1.
	\]
	Therefore, for $n \geq M$,
	\[
		|\lambda_n^-(x)|
		=
		1 +
		\frac{ \sigma \tr R_{nN+i}(x) - \sqrt{\discr R_{nN+i}(x)} }{ 2 a_{(n+1)N+i-1}}.
	\]
	We next claim the following holds true.
	\begin{claim} 
		\label{clm:1}
		There are $\delta' > 0$ and $M_0 \geq M$ such that for all $n \geq M_0$ and $x \in K$,
		\begin{equation}
			\label{eq:103}
			n \frac{\sigma \tr R_{nN+i}(x) - \sqrt{\discr R_{nN+i}(x)}} {a_{(n+1)N+i-1}} 
			\leq
			-1-\delta'.
		\end{equation}
	\end{claim}
	First observe that by the Stolz--Ces\'aro theorem and Proposition~\ref{prop:10}, we have
	\begin{equation}
		\label{eq:104}
		0 \leq 
		\lim_{n \to \infty} \frac{a_{(n+1)N+i-1}}{n} =
		\lim_{n \to \infty} \big( a_{(n+1)N+i-1} - a_{nN+i-1} \big) = 
		-\sigma \tr \calR_i(x).
	\end{equation}
	Since $(R_{nN+i} : n \in \NN)$ converges to $\calR_i$ uniformly on $K$,
	\begin{equation} 
		\label{eq:105}
		\lim_{n \to \infty} 
		\Big( \sigma \tr R_{nN+i}(x) - \sqrt{\discr R_{nN+i}(x)} \Big) =
		\sigma \tr \calR_i(x) - \sqrt{\discr \calR_i(x)}.
	\end{equation}
	Thus, by \eqref{eq:104} and \eqref{eq:105} we get
	\[
		\lim_{n \to \infty}
		n \frac{ \sigma \tr R_{nN+i}(x) - \sqrt{\discr R_{nN+i}(x)}}{a_{(n+1)N+i-1}} =
		\begin{cases}
			-\infty & \text{if } \tr \calR_i = 0, \\
			-1 - \tfrac{\sqrt{\discr \calR_i}}{-\sigma \tr \calR_i} & \text{otherwise,}
		\end{cases}
	\]
	which together with \eqref{eq:104} implies \eqref{eq:103}.
	
	Now, using Claim \ref{clm:1}, we conclude for all $n \geq M_0$,
	\[
		\sup_{x \in K}
		|\lambda^-_n(x)| \leq 1 - \frac{1+\delta'}{2n}.
	\]
	Consequently, by \eqref{eq:29}, there is $c' > 0$ such that for all $i' \in \{0, 1, \ldots, N-1\}$ and $n \geq M_0$,
	\begin{align*}
		\sup_{x \in K}{|u_{nN+i}(x)|} 
		&\leq c \prod_{j = n_0}^{n-1} \bigg(1 - \frac{1+\delta'}{2j} \bigg) \\
		&\leq
		c' \exp\bigg(-\frac{1+\delta'}{2} \log (n-1) \bigg).
	\end{align*}
	Hence,
	\[
		\sum_{n = 0}^\infty \sup_{x \in K}{|u_n(x)|^2} < \infty,
	\]
	thus by the proof of \cite[Theorem 5.3]{Silva2007} we conclude that $\sigmaEss(A) \cap K = \emptyset$. Since
	$K$ was any compact subset of $\RR \setminus \overline{\Lambda}$, we obtain
	$\sigmaEss(A) \subseteq \overline{\Lambda}$, and the theorem follows.
\end{proof}

\begin{remark}
	By \cite[Proposition 9]{PeriodicIII}, the regularity \eqref{eq:129} holds true for any compact set 
	$K \subset \RR$, and all $i \in \{0, 1, \ldots, N-1 \}$, if
	\[
		\bigg( \frac{\alpha_{n-1}}{\alpha_n} a_n - a_{n-1} : n \in \NN \bigg),
		\bigg( \frac{\beta_n}{\alpha_n} a_n - b_n : n \in \NN \bigg), 
		\bigg( \frac{1}{a_n} : n \in \NN \bigg) \in \calD_1^N(\RR).
	\]
\end{remark}

\section{Periodic modulations in non-Carleman setup} \label{sec:notCarleman}
In this section we shall consider Jacobi matrices such that
\[
	\sum_{n=0}^\infty \frac{1}{a_n} < \infty.
\]
Let us start with the following general observation.
\begin{proposition}
	\label{prop:3}
	Let $N$ be a positive integer and
	\[
		X_n(z) = \prod_{j=n}^{n+N-1} B_j(z).
	\]
	Let $K$ be a compact subset of $\CC$ containing $0$, and suppose that
	\begin{equation} 
		\label{eq:115}
		\sup_{n \geq 1} \sup_{z \in K} \| B_n(z) \| < \infty.
	\end{equation}
	Then there is $c > 0$ such that
	\[
		\sup_{x \in K} \| X_n(z) - X_n(0) \| \leq c \sum_{j = 0}^{N-1} \frac{1}{a_{n+j}}.
	\]
	In particular, if
	\begin{equation}
		\label{eq:116}
		\sum_{n=0}^\infty \frac{1}{a_n} < \infty,
	\end{equation}
	then
	\[
		\sum_{n=1}^\infty \sup_{z \in K} \| X_n(z) - X_n(0) \| < \infty.
	\]
\end{proposition}
\begin{proof}
	Let us notice that
	\[
		B_j(z) - B_j(0) =
		\begin{pmatrix}
			0 & 0 \\
			0 & \frac{z}{a_j}
		\end{pmatrix},
	\]
	thus
	\begin{equation}
		\label{eq:114}
		\| B_j(z) - B_j(0) \| \leq \frac{|z|}{a_j}.
	\end{equation}
	Since
	\[
		X_n(z) - X_n(0) =  
		\sum_{j=0}^{N-1} 
		\Bigg\{ \prod_{m=j+1}^{N-1} B_{n+m}(0) \Bigg\}
		\big( B_{n+j}(z) - B_{n+j}(0) \big)
		\Bigg\{ \prod_{m=0}^{j-1} B_{n+m}(z) \Bigg\},
	\]
	we have
	\[
		\| X_n(z) - X_n(0) \| \leq  
		\sum_{j=0}^{N-1} 
		\Bigg\{ \prod_{m=j+1}^{N-1} \big\| B_{n+m}(0) \big\| \Bigg\}
		\big\| B_{n+j}(z) - B_{n+j}(0) \big\|
		\Bigg\{ \prod_{m=0}^{j-1} \big\| B_{n+m}(z) \| \Bigg\}.
	\]
	Now the conclusion easily follows by \eqref{eq:115} and \eqref{eq:114}.
\end{proof}

The following corollary reproves the main result of \cite{Yafaev2019} obtained by a different method.
\begin{corollary}[Yafaev]
	\label{cor:3}
	Suppose that the Carleman's condition is \emph{not} satisfied and
	\begin{equation} \label{eq:134}
		\bigg( \frac{a_n}{\sqrt{a_{n-1} a_{n+1}}} - 1 : n \in \NN \bigg) \in \ell^1 \quad \text{and} \quad
		\bigg( \frac{b_n}{\sqrt{a_{n-1} a_n}} : n \in \NN \bigg) \in \calD_1.
	\end{equation}
	Let
	\begin{equation} \label{eq:136}
		q = \lim_{n \to \infty} \frac{b_n}{2 \sqrt{a_{n-1} a_n}}.
	\end{equation}
	If  $|q| \neq 1$, then for every $z \in \CC$ there is a basis $\{ u^+(z), u^-(z) \}$ of generalized eigenvectors
	associated with $z$ such that
	\begin{equation} 
		\label{eq:133}
		u^\pm_n(z) =
		\bigg(\prod_{j=1}^n \lambda_j^\pm(0) \bigg) \big( 1 + \epsilon^\pm_n(z) \big)
	\end{equation}
	where $\lambda^\pm_j(0)$ is the eigenvalue of $B_j(0)$, and
	$(\epsilon_n^\pm)$ is a sequence of holomorphic functions tending to zero uniformly on any compact subset of $\CC$.
\end{corollary}
\begin{proof}
	By \cite[Lemma 2.1]{Yafaev2019}
	\begin{equation} \label{eq:135}
		\bigg( \sqrt{\frac{a_{n+1}}{a_n}} : n \in \NN \bigg) \in \calD_1 \quad \text{and} \quad
		\lim_{n \to \infty} \sqrt{\frac{a_{n+1}}{a_n}} \geq 1.
	\end{equation}
	By Corollary~\ref{cor:2} and Lemma~\ref{lem:2} it implies
	\begin{equation}
		\label{eq:57}
		\bigg( \frac{a_{n-1}}{a_n} : n \in \NN \bigg) \in \calD_1.
	\end{equation}
	Next, we write
	\begin{equation} 
		\label{eq:137}
		\frac{b_n}{a_n} = \frac{b_n}{\sqrt{a_{n-1} a_n}} \sqrt{\frac{a_{n-1}}{a_n}},
	\end{equation}
	thus by \eqref{eq:134}, \eqref{eq:135} and Corollary~\ref{cor:1}
	\begin{equation}
		\label{eq:58}
		\bigg( \frac{b_n}{a_n} : n \in \NN \bigg) \in \calD_1.
	\end{equation}
	In particular, by \eqref{eq:57}  and \eqref{eq:58} we conclude that
	\[
		(B_n(0) : n \in \NN) \in \calD_1 \big( \GL(2, \RR) \big).
	\]
	Now, in view of Proposition~\ref{prop:3}
	\[
		B_n(z) = B_n(0) + E_n(z)
	\]
	where for any compact set $K \subset \CC$,
	\[
		\sum_{n=1}^\infty \sup_{z \in K} \| E_n(z) \|  < \infty.
	\]
	By \eqref{eq:57} and \eqref{eq:58}, there are $r, s \in \RR$,
	\[
		r = \lim_{n \to \infty} \frac{a_{n-1}}{a_n} \quad \text{and} \quad
		s = \lim_{n \to \infty} \frac{b_{n}}{a_n}.
	\]
	Then the limit of $(B_n(0) : n \in \NN)$ is
	\[
		\calB = 
		\begin{pmatrix}
			0 & 1  \\
			-r & -s
		\end{pmatrix}.
	\]
	Notice that
	\[
		\discr \calB = s^2 - 4r = 
		r \Big(\frac{s}{2 \sqrt{r}} - 1 \Big) 
		\Big(\frac{s}{2 \sqrt{r}} + 1 \Big).
	\]
	On the other hand, by \eqref{eq:135}, \eqref{eq:136} and \eqref{eq:137}, we can easily deduce that
	\[
		r \in (0, 1] \quad \text{and} \quad q = \frac{s}{2 \sqrt{r}}.
	\]
	Therefore, $\discr \calB \neq 0$ whenever $\abs{q} \neq 1$. Fix a compact set $K \subset \CC$. 
	If $\discr \calB > 0$, then $\calB$ has two eigenvectors
	\[
		v^+ = 
		\begin{pmatrix}
			1 \\
			\lambda^+
		\end{pmatrix}
		\qquad
		v^- =
		\begin{pmatrix}
			1 \\
			\lambda^-
		\end{pmatrix}
	\]
	corresponding to the eigenvalues
	\[
		\lambda^+ = \frac{-s + \sqrt{s^2-4r}}{2}, \qquad
		\lambda^- = \frac{-s - \sqrt{s^2-4r}}{2}.
	\]
	Since $\det \calB = r$ these eigenvalues are non-zero.
	Let us consider the system
	\[
		\Phi_{n+1} = \big( B_n(0) + E_n \big) \Phi_n.
	\]
	By Corollary \ref{cor:4}, there is a sequence of mappings $(\Phi^\pm_n : n \geq n_0)$ so that
	\begin{equation}
		\label{eq:36}
		\lim_{n \to \infty} \sup_{z \in K}{\bigg\|\frac{\Phi^\pm_n(z)}{\prod_{j=1}^{n-1} \lambda^\pm_j(0)} - v^\pm
		\bigg\|} = 0.
	\end{equation}
	Since $B_n$ is invertible for any $n$, we set
	\[
		\phi^\pm_1 = B_1^{-1} \cdots B_{n_0}^{-1} \Phi^\pm_{n_0}.
	\]
	Then for $n \geq 1$, we define
	\[
		\phi^\pm_{n+1} = B_n \phi_n^\pm.
	\]
	Finally, to obtain a generalized eigenvector associated with $z \in K$, we set
	\[
		u_n^\pm(z) = 
		\begin{cases}
			\langle \phi^\pm_1(z), e_1\rangle & \text{if } n = 0,\\
			\langle \phi^\pm_n(z), e_2\rangle & \text{if } n \geq 1.
		\end{cases}
	\]
	Now, by \eqref{eq:36} we easily see that
	\[
		u^\pm_n(z) = 
		\bigg(\prod_{j = 1}^{n-1} \lambda_j^\pm(0) \bigg)
		\big(\lambda^\pm + \epsilon^\pm_n(z)\big)
	\]
	with 
	\[
		\lim_{n \to \infty} \sup_{z \in K}{|\epsilon^\pm_n(z)|} = 0.
	\]
	Since $(\lambda^\pm_j(0))$ converges to $\lambda^\pm$, we obtain \eqref{eq:133}.
	When $\discr \calB < 0$, the reasoning is analogous.
\end{proof}

\subsection{Perturbation of the identity}
\begin{theorem}
	\label{thm:8a}
	Let $N$ be a positive integer. Let $A$ be a Jacobi matrix with $N$-periodically modulated entries so that
	$\frakX_0(0) = \sigma \Id$ for a certain $\sigma \in \{-1, 1 \}$. Assume that there are $i \in \{0, 1, \ldots, N-1 \}$,
	and a sequence of positive numbers $(\gamma_n : n \in \NN_0)$ satisfying
	\[
		\sum_{n=0}^\infty \frac{1}{\gamma_n} = \infty,
	\]
	such that $R_{nN+i}(0) = \gamma_n \big(X_{nN+i}(0) - \sigma \Id \big)$ converges to the non-zero matrix $\calR_i$.
	If $\big( R_{nN+i}(0) : n \in \NN \big)$ belongs to $\calD_1 \big( \Mat(2, \RR) \big)$, $\discr \calR_i > 0$, and
	\begin{equation}
		\label{eq:45}
		\sum_{n=0}^\infty \frac{1}{a_n} < \infty,
	\end{equation}
	then $A$ is self-adjoint if and only if there is $n_0 \geq 1$, such that
	\begin{equation} \label{eq:72}
		\sum_{n=n_0}^{\infty} 
		\prod_{j=n_0}^n 
		\bigg| 1 + \frac{\sigma \tr R_{jN+i}(0) + \sqrt{\discr R_{jN+i}(0)}}{2 \gamma_j} \bigg|^2 = \infty.
	\end{equation}
	Moreover, if $A$ is self-adjoint, then $\sigmaEss(A) = \emptyset$.
\end{theorem}
\begin{proof}
	Since $\discr \calR_i > 0$, there are $\delta > 0$ and $n_0 \in \NN$, such that for all $j \geq n_0$,
	\[
		\discr R_{jN+i}(0) > \delta.
	\]
	Hence, the matrix $R_{jN+i}(0)$ has two distinct eigenvalues
	\[
		\xi^+_j(0) = \frac{\tr R_{jN+i}(0) + \sigma \sqrt{\discr R_{jN+i}(0)}}{2},
		\qquad\text{and}\qquad
		\xi^-_j(0) = \frac{\tr R_{jN+i}(0) - \sigma \sqrt{\discr R_{jN+i}(0)}}{2},
	\]
	thus the matrix $X_{jN+i}(0)= \sigma \Id + \gamma_j^{-1}R_{jN+i}(0)$ has two eigenvalues
	\[
		\lambda^+_j(0) = \sigma + \frac{\xi^+_j(0)}{\gamma_j},
		\qquad\text{and}\qquad
		\lambda^-_j(0) = \sigma + \frac{\xi^-_j(0)}{\gamma_j}.
	\]
	Let $K$ be any compact subset of $\RR$. By Proposition~\ref{prop:3}, we can write
	\[
		X_{nN+i}(x) = \sigma \Id + \frac{1}{\gamma_n} R_{nN+i}(0) + E_{nN+i}(x)
	\]
	where
	\[
		\sum_{n=0}^\infty \sup_{x \in K} \| E_{nN+i}(x) \| < \infty.
	\]
	Since $(R_{jN+i}(0) : j \in \NN)$ belongs to $\calD_1\big(\Mat(2, \RR)\big)$, by Corollary~\ref{cor:5},
	there are two sequences $\big( \Phi_j^- : j \geq n_0 \big)$ and $\big( \Phi_j^+ : j \geq n_0 \big)$ satisfying
	\[
		\Phi_{j+1} = \big( X_{jN+i}(0) + E_{jN+i}\big) \Phi_j,
	\]
	and such that
	\begin{equation} \label{eq:140}
		\lim_{n \to \infty} \sup_K
		\bigg\| \frac{\Phi_n^\pm}{\prod_{j = n_0}^{n-1} \lambda^\pm_j(0)} - v^\pm \bigg\| = 0
	\end{equation}
	for certain $v^-, v^+ \neq 0$. Let
	\[
		\phi_1^\pm = B_1^{-1} B_2^{-1} \cdots B_{n_0}^{-1} \Phi_{n_0}^\pm.
	\]
	For $n \geq 1$, we set
	\[
		\phi_{n+1}^\pm = B_n \phi_n^\pm.
	\]
	Then for $kN+i' > n_0N+i$ with $i' \in \{0, 1, \ldots, N-1\}$, we have
	\[
		\phi_{kN+i'}^\pm = 
		\begin{cases}
			B_{kN+i'}^{-1} B_{kN+i'+1}^{-1} \cdots B_{kN+i-1}^{-1} \Phi_k^\pm
			&\text{if } i' \in \{0, 1, \ldots, i-1\}, \\
			\Phi_k^\pm & \text{if } i ' = i,\\\
			B_{kN+i'-1} B_{kN+i'-2} \cdots B_{kN+i} \Phi_k^\pm &
			\text{if } i' \in \{i+1, \ldots, N-1\}.
		\end{cases}
	\]
	Consequently, we obtain
	\begin{equation} 
		\label{eq:71}
		\lim_{k \to \infty} \frac{\phi_{kN+i'}^\pm}{\prod_{j = n_0}^n \lambda^\pm_j(0)} = 
		\begin{cases}
			\frakB_{i'}^{-1}(0) \frakB_{i'+1}^{-1}(0) \cdots \frakB_{i-1}^{-1}(0) v^\pm	
			&\text{if } i' \in \{0, 1, \ldots, i-1\} \\
			v^\pm &\text{if } i' = i, \\
			\frakB_{i'-1}(0) \frakB_{i'-2}(0) \cdots \frakB_{i}(0) v^\pm
			&\text{if } i' \in \{i+1, \ldots ,N-1\},
		\end{cases}
	\end{equation}
	uniformly on $K$. Let
	\[
		u_n^\pm(x) =
		\begin{cases}
			\langle \phi_1^\pm(x), e_1 \rangle & \text{if } n = 0, \\
			\langle \phi_n^\pm(x), e_2 \rangle & \text{if } n \geq 1.
		\end{cases}
	\]
	Then $(u_n^+(x) : n \in \NN_0)$ and $(u_n^-(x) : n \in \NN_0)$ are two generalized eigenvectors associated with
	$x \in K$. Since their asymptotic behavior is different from each other, they are linearly independent. 
	
	By \eqref{eq:140} and \eqref{eq:71}, there is a constant $c>0$ such that for all $n > n_0$, and $x \in K$,
	\begin{equation}
		\label{eq:44}
		\big| u^{\pm}_{nN+i}(x) \big|^2 + \big| u^{\pm}_{nN+i-1}(x) \big|^2
		=
		\big\| \phi^{\pm}_{nN+i}(x) \big\|^2
		\geq
		c \prod_{j = n_0}^{n-1} \big| \lambda_j^{\pm}(0) \big|^2.
	\end{equation}
	Moreover, for all $n > n_0$, $i' \in \{0, 1, \ldots, N-1\}$, and $x \in K$,
	\begin{equation}
		\label{eq:48a}
		\big\| \phi^\pm_{nN+i'}(x) \big\|^2 \leq 
		c \prod_{j = n_0}^{n-1} \big| \lambda_j^{\pm}(0) \big|^2.
	\end{equation}	
	Since $|\lambda_{j}^{-}(0)| \leq |\lambda_{j}^{+}(0)|$, we obtain
	\begin{equation}
		\label{eq:73}
		\sum_{n=n_0+1}^\infty |u_n^-(x)|^2 \leq c
		\sum_{n=n_0+1}^\infty |u_n^+(x)|^2.
	\end{equation}
	Now, observe that if \eqref{eq:72} is satisfied then by \eqref{eq:44} the generalized eigenvector
	$(u_n^+(x) : n \in \NN_0)$ is not square-summable, hence by \cite[Theorem 6.16]{Schmudgen2017}, the operator
	$A$ is self-adjoint. On the other-hand, if \eqref{eq:72} is not satisfied, then by \eqref{eq:48a} and \eqref{eq:73}, 
	all generalized eigenvectors are square-summable, thus by \cite[Theorem 6.16]{Schmudgen2017}, the operator
	$A$ is not self-adjoint.

	Finally, let us suppose that $A$ is self-adjoint. By the proof of \cite[Theorem 5.3]{Silva2007}, if
	\begin{equation} 
		\label{eq:70}
	    \sum_{n=0}^\infty \sup_{x \in K} | u^-_n(x) |^2 < \infty,
	\end{equation}
	then $\sigmaEss(A) \cap K = \emptyset$, and since $K$ is any compact subset of $\RR$ this implies that 
	$\sigmaEss(A) = \emptyset$. Therefore, to complete the proof of the theorem it is enough to show \eqref{eq:70}.
	Observe that $E_{nN+i}(0)=0$, thus 
	\[
		\lambda_j^+(0) \lambda_j^-(0) = 
		\det X_{jN+i}(0) = \frac{a_{jN+i-1}}{a_{(j+1)N+i-1}},
	\]
	and so
	\[
		\prod_{j=n_0}^k
		\lambda^-_j(0) \lambda^+_j(0)
		=
		\frac{a_{n_0 N+i-1}}{a_{(k+1)N+i-1}}.
	\]
	Consequently, by \eqref{eq:45},
	\[
		\sum_{k=n_0}^\infty 
		\prod_{j=n_0}^k
		\big| \lambda^-_j(0) \lambda^+_j(0) \big| 
		=
		\sum_{k=n_0}^\infty \frac{a_{n_0 N+i-1}}{a_{(k+1)N+i}} < \infty,
	\]
	which together with $| \lambda^-_j(0) | \leq | \lambda^+_j(0) |$ implies that
	\[
		\sum_{k=n_0}^\infty 
		\prod_{j=n_0}^k
		\big| \lambda^-_j(0) \big|^2 
		< \infty.
	\]
	Hence, by \eqref{eq:48a} we obtain \eqref{eq:70}, and the theorem follows.
\end{proof}

By the similar reasoning one can prove the following theorem.
\begin{theorem} 
	\label{thm:8}
	Let $N$ be a positive integer. Let $A$ be a Jacobi matrix with $N$-periodically
	modulated entries so that $\frakX_0(0) = \sigma \Id$ for a certain $\sigma \in \{-1, 1 \}$. Assume that there are 
	$i \in \{0, 1, \ldots, N-1 \}$ and a sequence of positive numbers $(\gamma_n : n \in \NN_0)$ such that
	$R_{nN+i}(0) = \gamma_n \big(X_{nN+i}(0) - \sigma \Id\big)$ converges to the non-zero matrix $\calR_i$.
	If $\big( R_{nN+i}(0) : n \in \NN \big)$ belongs to $\calD_1 \big( \Mat(2, \RR) \big)$, $\discr \calR_i < 0$, 
	and 
	\[
		\sum_{n=0}^\infty \frac{1}{a_n} < \infty,
	\]
	then the operator $A$ is \emph{not} self-adjoint.
\end{theorem}

Proposition~\ref{prop:7} motivates us to the following notion.
Given a sequence $(w_n : n \in \NN)$ such that $w_n > 0$ for all $n \in \NN$, we introduce the
weighted Stolz class. We say that $(x_n)$ a bounded sequence from a normed vector space $X$ belongs to
$\calD_1(X; w)$, if
\[
	\sum_{n = 1}^\infty \left\|\Delta x_n \right\| w_n < \infty.
\]
Moreover, given a positive integer $N$, we say that $x \in \calD^N_{1}(X; w)$ if for each 
$i \in \{0, 1, \ldots, N-1 \}$,
\[
	 \big( x_{nN+i} : n \in \NN \big) \in \calD_{1}(X; w).
\]

Similar reasoning to \cite[Corollary 1]{SwiderskiTrojan2019} leads to the following fact.
\begin{proposition}
	\label{prop:4}
	For any weight $(w_n)$
	\begin{enumerate}[(i), leftmargin=2em]
		\item If $(x_n), (y_n) \in \calD_1(X; w)$, then $(x_n y_n), (x_n + y_n) \in \calD_1(X;w)$.
		\item If $(x_n) \in \calD_1(K, \CC; w)$, and $\|x_n(t)\| \geq c > 0$ for all $n \geq \NN_0$ and $t \in K$, then
		$(x_n^{-1}) \in \calD_1(K, \CC; w)$.
	\end{enumerate}
\end{proposition}

The following proposition describes a way to construct matrices $(R_n : n \in \NN)$ satisfying the hypotheses of
Theorems \ref{thm:8a} and \ref{thm:8}.
\begin{proposition}
	\label{prop:2}
	Let $N$ be a positive integer and $i \in \{0, 1, \ldots, N-1\}$. Let $A$ be a Jacobi matrix with $N$-periodically
	modulated entries so that $\frakX_0(0) = \sigma \Id$ for a certain $\sigma \in \{-1, 1 \}$. Assume that there is
	$(\gamma_n : n \in \NN_0)$ a sequence of positive numbers such that
	$R_{nN+i}(0) = \gamma_n\big(X_{nN+i}(0) - \sigma \Id\big)$ converges to non-zero matrix $\calR_i$.
	If there are two $N$-periodic sequences $(\tilde{s}_{i'} : i' \in \NN_0)$ and 
	$(\tilde{z}_{i'} : i' \in \NN_0)$ such that
	\begin{equation}
		\label{eq:117}
		\tilde{s}_{i'} = 
		\lim_{n \to \infty} \gamma_n \Big(\frac{\alpha_{i'-1}}{\alpha_{i'}} - \frac{a_{nN+i'-1}}{a_{nN+i'}} \Big),
		\qquad\text{and}\qquad
		\tilde{z}_{i'} = 
		\lim_{n \to \infty} \gamma_n \Big(\frac{\beta_{{i'}}}{\alpha_{i'}} - \frac{b_{nN+i'}}{a_{nN+i'}} \Big),
	\end{equation}
	then
	\begin{equation} 
		\label{eq:118}
		\calR_i
		=
		\sum_{j=0}^{N-1}
		\left\{ \prod_{m=j+1}^{N-1} \frakB_{i+m}(0) \right\}
		\begin{pmatrix}
			0 & 0 \\
			\tilde{s}_{i+j} & \tilde{z}_{i+j}
		\end{pmatrix}
		\left\{ \prod_{m=0}^{j-1} \frakB_{i+m}(0) \right\}
	\end{equation}
	and
	\begin{equation} \label{eq:118a}
		\tr \calR_i = -\sigma \sum_{j=0}^{N-1} \tilde{s}_{i+j} \frac{\alpha_{i+j}}{\alpha_{i+j-1}}.
	\end{equation}
	Moreover, if there is a weight $(w_n : n \in \NN)$ so that for all $i' \in \{0, 1, \ldots, N-1 \}$,
	\begin{equation} 
		\label{eq:120}
		\bigg( \frac{1}{\gamma_n} : n \in \NN \bigg), 
		\bigg( \gamma_n \Big(\frac{\alpha_{i'-1}}{\alpha_{i'}} - \frac{a_{nN+i'-1}}{a_{nN+i'}} \Big) : n \in \NN \bigg),
		\bigg( \gamma_n \Big(\frac{\beta_{i'}}{\alpha_{i'}} - \frac{b_{nN+i'}}{a_{nN+i'}} \Big) : n \in \NN \bigg) 
		\in \calD_1(\RR; w),
	\end{equation}
	then
	\begin{equation} 
		\label{eq:121}
		\big( R_{nN+i}(0) : n \in \NN \big) \in \calD_1 \big( \Mat(2, \RR); w\big).
	\end{equation}
\end{proposition}
\begin{proof}
	Since
	\[
		X_n(0) - \frakX_n(0) =  
		\sum_{i'=0}^{N-1} 
		\left\{ \prod_{m=i'+1}^{N-1} \frakB_{n+m}(0) \right\}
		\big( B_{n+i'}(0) - \frakB_{n+i'}(0) \big)
		\left\{ \prod_{m=0}^{i'-1} B_{n+m}(0) \right\},
	\]
	and $\frakX_i(0) = \sigma \Id$, we get
	\begin{equation} 
		\label{eq:119}
		R_{nN+i}(0)
		=
		\sum_{i'=0}^{N-1} 
		\left\{ \prod_{m=i'+1}^{N-1} \frakB_{i+m}(0) \right\}
		\gamma_n 
		\big( B_{nN+i+i'}(0) - \frakB_{i+i'}(0) \big)
		\left\{ \prod_{m=0}^{i'-1} B_{nN+i+m}(0) \right\}.
	\end{equation}
	Observe that
	\begin{equation} 
		\label{eq:122}
		\gamma_n \big( B_{nN+i+i'}(0) - \frakB_{i+i'}(0) \big) 
		=
		\gamma_n
		\begin{pmatrix}
			0 & 0 \\
			\frac{\alpha_{i+i'-1}}{\alpha_{i+i'}} - \frac{a_{nN+i+i'-1}}{a_{nN+i+i'}} &
			\frac{\beta_{i+i'}}{\alpha_{i+i'}} - \frac{b_{nN+i+i'}}{a_{nN+i+i'}}
		\end{pmatrix},
	\end{equation}
	thus by \eqref{eq:117} we obtain
	\[
		\lim_{n \to \infty}
		\gamma_n \big( B_{nN+i+i'}(0) - \frakB_{i+i'}(0) \big)
		=
		\begin{pmatrix}
			0 & 0 \\
			\tilde{s}_{i+i'} & \tilde{z}_{i+i'}
		\end{pmatrix}.
	\]
	Now, \eqref{eq:118} easily follows from \eqref{eq:119}. The proof of \eqref{eq:118a} is analogous to Proposition
	\ref{prop:10}, cf. \eqref{eq:101} and \eqref{eq:18}.

	We proceed to showing \eqref{eq:121}. By \eqref{eq:120}, for each $i' \in \{0, 1, \ldots, N-1 \}$,
	\[
		\bigg( \frac{a_{nN+i'-1}}{a_{nN+i'}} : n \in \NN \bigg), 
		\bigg( \frac{b_{nN+i'}}{a_{nN+i'}} : n \in \NN \bigg) \in \calD_1(\RR; w),
	\]
	thus,
	\begin{equation} 
		\label{eq:123}
		\big( B_{nN+i'}(0) : n \in \NN \big) \in \calD_1 \big( \GL(2, \RR); w\big).
	\end{equation}
	Moreover, in view of \eqref{eq:122}, the condition \eqref{eq:120} implies that
	\begin{equation}
		\label{eq:124}
		\Big( \gamma_n \big( B_{nN+i+i'}(0) - \frakB_{i+i'}(0) \big) : n \in \NN \Big) 
		\in \calD_1 \big( \Mat(2, \RR); w\big).
	\end{equation}
	Now, \eqref{eq:123} and \eqref{eq:124} together with \eqref{eq:119} implies \eqref{eq:121}
\end{proof}

\subsection{A periodic modulations of Kostyuchenko--Mirzoev's class}
\label{sec:KM}
Let $N$ be a positive integer. We say that a Jacobi matrix $A$ associated to $(a_n : n \in \NN_0)$
and $(b_n: n \in \NN_0)$ belongs to $N$-periodically modulated Kostyuchenko--Mirzoev's class, if there are two
$N$-periodic sequences $(\alpha_n : n \in \ZZ)$ and $(\beta_n : n \in \ZZ)$ of positive and real numbers, respectively,
such that
\[
	a_n = \alpha_n \tilde{a}_n \Big( 1 + \frac{f_n}{\gamma_n} \Big) > 0, \qquad\text{and}\qquad
	b_n = \frac{\beta_n}{\alpha_n} a_n
\]
where $(f_n : n \in \NN_0)$ is bounded sequence, and $(\tilde{a}_n : n \in \NN_0)$ is a positive sequence satisfying
\[
	\sum_{n=0}^\infty \frac{1}{\tilde{a}_n} < \infty \qquad \text{and} \qquad
	\lim_{n \to \infty} \gamma_n \Big( 1- \frac{\tilde{a}_{n-1}}{\tilde{a}_n} \Big) = \kappa > 0
\]
for a certain positive sequence $(\gamma_n : n \in \NN_0)$ tending to infinity.

This class contains interesting examples of Jacobi matrices giving rise to self-adjoint operators which 
do not satisfy the Carleman's condition. Moreover, we formulate certain conditions under which 
the essential spectrum is empty. This class has been studied in
\cite{Kostyuchenko1999, JanasMoszynski2003, Silva2004, Silva2007, Silva2007a} in the case when $N$ is an even integer, 
$\alpha_n \equiv 1, \beta_n \equiv 0$, and
\[
	\tilde{a}_n = (n+1)^\kappa, \qquad \gamma_n = n+1
\]
for some $\kappa > 1$.

\begin{theorem} 
	\label{thm:4}
	Let $N$ be a positive integer. Let $A$ be a Jacobi matrix from $N$-periodically modulated 
	Kostyuchenko--Mirzoev's class so that $\frakX_0(0) = \sigma \Id$ for a certain $\sigma \in \{-1, 1\}$.
	Suppose that there is a weight $(w_n : n \in \NN)$, so that
	\begin{equation}
		\label{eq:51}
		\bigg( \gamma_n \Big( 1 - \frac{\tilde{a}_{n-1}}{\tilde{a}_n} \Big) : n \in \NN \bigg),
		\big( f_n : n \in \NN \big),
		\bigg( \frac{\gamma_{n-1}}{\gamma_n} : n \in \NN \bigg) \in \calD_1^N(\RR; w),
	\end{equation}
	and
	\begin{equation} \label{eq:51a}
		\lim_{n \to \infty} \frac{\gamma_{n-1}}{\gamma_n} = 1.
	\end{equation}
	Then for all $i \in \{0, 1, \ldots, N-1\}$, the matrices
	$R_{nN+i}(0) = \gamma_{nN} \big(X_{nN+i}(0) - \sigma \Id\big)$ converge to the non-zero matrix $\calR_i$, 
	\begin{equation}
		\label{eq:33}
		\big( R_{nN+i}(0) : n \in \NN \big) \in \calD_1  \big(\Mat(2, \RR); w\big),
	\end{equation}
	and
	\[
		\sum_{n = 0}^\infty
		\sup_{z \in K}{
		\big\|X_{nN+i}(z) - X_{nN+i}(0) \big\|
		}
		< \infty
	\]
	for every compact set $K \subset \CC$. Moreover, $\tr \calR_i = -\kappa \sigma N$, and
	\begin{equation}
		\label{eq:31}
		\calR_i 
		=
		\sum_{j=0}^{N-1} 
		\frac{\alpha_{i+j-1}}{\alpha_{i+j}} \big( \kappa + \frakf_{i+j} - \frakf_{i+j-1} \big)
		\left\{ \prod_{m=j+1}^{N-1} \frakB_{i+m}(0) \right\}
		\begin{pmatrix}
			0 & 0 \\
			1 & 0
		\end{pmatrix}
		\left\{ \prod_{m=0}^{j-1} \frakB_{i+m}(0) \right\}
	\end{equation}
	where $(\frakf_n : n \in \ZZ)$ is $N$-periodic sequence so that
	\begin{equation} \label{eq:63}
		\lim_{n \to \infty} |f_n - \frakf_n| = 0.
	\end{equation}
\end{theorem}
\begin{proof}
	To prove \eqref{eq:33}, we are going to apply Proposition \ref{prop:2}.
	To do so, we need to check \eqref{eq:120}. In fact, it is enough to show that for any $i \in \{0, 1, \ldots, N-1 \}$,
	\begin{equation}
		\label{eq:61}
		\bigg( \gamma_{jN} \Big( \frac{\alpha_{i-1}}{\alpha_{i}} - \frac{a_{jN+i-1}}{a_{jN+i}} \Big) : j \in \NN \bigg)
		\in \calD_1(\RR; w).
	\end{equation}
	We write
	\[
		\gamma_{jN} \Big( 
		\frac{\alpha_{i-1}}{\alpha_{i}} - 
		\frac{a_{jN+i-1}}{a_{jN+i}} 
		\Big) = 
		\frac{\alpha_{i-1}}{\alpha_i} 
		\gamma_{jN}
		\Big(
		1 - 
		\frac{\tilde{a}_{jN+i-1}}{\tilde{a}_{jN+i}} \Big( 1 + \frac{e_{j}}{\gamma_{jN}} \Big)
		\bigg)
	\]
	where
	\begin{equation} \label{eq:65}
		e_{j} = 
		\gamma_{jN} \bigg( \frac{1+\tfrac{f_{jN+i-1}}{\gamma_{jN+i-1}}}{1+\tfrac{f_{jN+i}}{\gamma_{jN+i}}} -1 \bigg) = 
		\frac{\gamma_{jN}}{\gamma_{jN+i-1}} \frac{f_{jN+i-1}- f_{jN+i} \tfrac{\gamma_{jN+i-1}}{\gamma_{jN+i}}}
		{1+\tfrac{f_{jN+i}}{\gamma_{jN+i}}}.
	\end{equation}
	Thus
	\begin{equation} \label{eq:66}
		\gamma_{jN} \Big( 
		\frac{\alpha_{i-1}}{\alpha_{i}} - 
		\frac{a_{jN+i-1}}{a_{jN+i}} 
		\Big) =
		\frac{\alpha_{i-1}}{\alpha_i} 
		\gamma_{jN} \Big(
		1 - 
		\frac{\tilde{a}_{jN+i-1}}{\tilde{a}_{jN+i}} \Big)
		- \frac{\alpha_{i-1}}{\alpha_i}
		\frac{\tilde{a}_{jN+i-1}}{\tilde{a}_{jN+i}}
		e_j
	\end{equation}
	and by \eqref{eq:51} we easily obtain \eqref{eq:61}.
	
	In view of \eqref{eq:51a} and \eqref{eq:63}, the formula \eqref{eq:65} gives
	\[
		\lim_{j \to \infty} e_{jN+i} = \frakf_{i-1} - \frakf_i.
	\]
	Thus, by \eqref{eq:66} 
	\[
		\lim_{j \to \infty} 
		\gamma_{jN} \Big( 
		\frac{\alpha_{i-1}}{\alpha_{i}} - 
		\frac{a_{jN+i-1}}{a_{jN+i}} 
		\Big) = \frac{\alpha_{i-1}}{\alpha_i} (\kappa + \frakf_{i} - \frakf_{i-1}).
	\]
	Fix a compact set $K \subset \CC$. Since the condition \eqref{eq:116} is satisfied, by Proposition \ref{prop:3},
	for all $z \in K$,
	\[
		X_{jN+i}(z) = \sigma \Id + \frac{1}{\gamma_j} R_{jN+i}(0) + E_{jN+i}(z)
	\]
	where
	\[
		\sup_{z \in K} \|E_n(z)\| \leq \frac{c}{\tilde{a}_n}.
	\]
	Finally, by Proposition \ref{prop:2}, we obtain \eqref{eq:33} and \eqref{eq:31}. 
\end{proof}

\subsubsection{Examples of modulated sequences}
In this section we present examples of sequences $(\tilde{a}_n : n \in \NN_0)$ and 
$(\gamma_n : n \in \NN_0)$ satisfying the assumptions of Theorem~\ref{thm:4}.

\begin{example}
Let $\kappa > 1$ and 
\[
	\tilde{a}_n = (n+1)^{\kappa} \quad \text{and} \quad 
	\gamma_n = n+1.
\]
Then
\[
	\gamma_n \Big( 1 - \frac{\tilde{a}_{n-1}}{\tilde{a}_n} \Big) = 
	\kappa + \frac{\kappa (\kappa-1)}{2n} + \calO \Big( \frac{1}{n^2} \Big).
\]
\end{example}

\begin{example}
Let
\[
	\tilde{a}_n = (n+1) \log^2(n+2) \quad \text{and} \quad
	\gamma_n = n+1.
\]
Then
\[
	\gamma_n \Big( 1 - \frac{\tilde{a}_{n-1}}{\tilde{a}_n} \Big) = 
	1 + \frac{2}{\log n} - \frac{3}{n \log n} + \calO \Big( \frac{1}{n \log^2 n} \Big).
\]
\end{example}

\begin{proposition} 
	\label{prop:5}
	Suppose that the hypotheses of Theorem~\ref{thm:4} are satisfied with $\gamma_n = n+1$. Assume that
	$\discr \calR_0 > 0$. Then 
	\begin{enumerate}[(i), leftmargin=2em]
		\item
		\label{cas:1}
		if $-\kappa + \frac{1}{N} \sqrt{\discr \calR_0} > -1$, then the operator $A$ is self-adjoint;
		\item 
		\label{cas:2}
		if $-\kappa + \frac{1}{N} \sqrt{\discr \calR_0} < -1$, then the operator $A$ is not self-adjoint.
	\end{enumerate}
	Moreover, if the operator $A$ is self-adjoint then $\sigmaEss(A) = \emptyset$.
\end{proposition}
\begin{proof}
	We shall consider the case \ref{cas:1} only as the reasoning in \ref{cas:2} is similar. 
	By Theorem~\ref{thm:8a} it is enough to check whether there is $n_0 \geq 1$ so that the series
	\begin{equation} 
		\label{eq:75}
		\sum_{n=n_0}^\infty 
		\prod_{j=n_0}^n 
		\bigg| 1 + \frac{\sigma \tr R_{jN}(0) + \sqrt{\discr R_{jN}(0)}}{2 jN} \bigg|^2
	\end{equation}
	diverges. Let us select $\delta > 0$ so that
	\begin{equation}
		\label{eq:47}
		- \kappa + \frac{1}{N} \sqrt{\discr \calR_0} - \delta > -1.
	\end{equation}
	By Theorem \ref{thm:4}, $\tr \calR_0 = -\kappa \sigma N$. Hence, there is $j_0 \in \NN$ such that for all 
	$j \geq j_0$,
	\[
		\Big|
		\big( \sigma \tr R_{jN}(0) + \sqrt{\discr R_{jN}(0)}\big) 
		-
		\big( -\kappa N + \sqrt{\discr \calR_0}\big)\Big| \leq N \delta.
	\]
	Thus,
	\[
		1 + \frac{\sigma \tr R_{jN}(0) + \sqrt{\discr R_{jN}(0)}}{2 jN}
		\geq 1 + \frac{1}{2jN}\big(-\kappa N + \sqrt{\discr \calR_0} - N \delta\big),
	\]
	and so
	\begin{align*}
		\log \bigg( \prod_{j=j_0}^n \bigg| 1 + \frac{\sigma \tr R_{jN}(0) + \sqrt{\discr R_{jN}(0)}}{2 jN} \bigg| \bigg)
		&\geq
		-c + \frac{1}{2} \Big(-\kappa + \frac{1}{N} \sqrt{\discr \calR_0} - \delta \Big) \sum_{j = 1}^n \frac{1}{j} \\
		&\geq
		-c' + \frac{1}{2} \Big( -\kappa + \frac{1}{N} \sqrt{\discr \calR_0} - \delta \Big) \log n.
	\end{align*}
	Therefore,
	\[
		\prod_{j=j_0}^n
		\bigg| 1 + \frac{\sigma \tr R_{jN}(0) + \sqrt{\discr R_{jN}(0)}}{2 jN} \bigg|^2 \geq
		c n^{-\kappa + \frac{1}{N} \sqrt{\discr \calR_0} - \delta}
	\]
	which, in view of \eqref{eq:47}, implies that the series \eqref{eq:75} is divergent.
\end{proof}

\begin{example}
	For $0 < \tau < 1$ we set
	\[
		\tilde{a}_n = \ue^{n^\tau}
		\qquad\text{and}\qquad
		\gamma_n = \max\{ n^{1-\tau}, 1\}.
	\]
	Let $m \in \NN$ be chosen so that
	\[
		1 - \frac{1}{m-2} \leq \tau < 1 - \frac{1}{m-1}.
	\]
	Then
	\begin{align*}
		1-\frac{\tilde{a}_{n-1}}{\tilde{a}_n}
		&=
		\sum_{j = 1}^{m-1} 
		\frac{(-1)^{j+1}}{j!} \big(n^\tau - (n-1)^\tau \big)^j + \calO \big(n^{m(\tau-1)}\big) \\
		&=
		\sum_{j = 1}^{m-1} \frac{(-1)^{j+1}}{j!} n^{\tau j}
		\Big(1 - (1-n^{-1})^\tau\Big)^j + \calO\big(n^{m(\tau-1)}\big).
	\end{align*}
	Since
	\[
		1 - (1-n^{-1})^\tau 
		= \tau n^{-1} - \tfrac{\tau(\tau-1)}{2} n^{-2} + \calO\big(n^{-3}\big),
	\]
	we obtain
	\begin{align*}
		1-\frac{\tilde{a}_{n-1}}{\tilde{a}_n}
		&=
		n^{\tau}\Big(\tau n^{-1} - \tfrac{\tau(\tau-1)}{2} n^{-2} + \calO\big(n^{-3}\big) \Big)
		-\sum_{j = 2}^{m-1} \frac{(-1)^j}{j!} n^{\tau j} \Big(\tau n^{-1} + \calO\big(n^{-2}\big)\Big)^j 
		+ \calO\big(n^{m(\tau-1)}\big)\\
		&=
		\tau n^{\tau-1} - \tfrac{\tau(\tau-1)}{2} n^{\tau-2} + \calO\big(n^{\tau-3}\big)
		-
		\sum_{j = 2}^{m-1} \frac{(-1)^j}{j!} n^{\tau j} \Big(\tau^j n^{-j} + \calO\big(n^{-j-1}\big)\Big)
		+\calO\big(n^{m(\tau-1)}\big)\\
		&=
		\tau n^{\tau-1} - \tfrac{\tau(\tau-1)}{2} n^{\tau-2} 
		-
		\sum_{j = 2}^{m-1} \frac{(-\tau)^j}{j!} n^{j(\tau-1)} + \calO\big(n^{2\tau-3}\big)
		+\calO\big(n^{m(\tau-1)}\big).
	\end{align*}
	Hence,
	\[
		\gamma_n \bigg(1- \frac{\tilde{a}_{n-1}}{\tilde{a}_n}\bigg)
		=
		\tau + \tfrac{\tau(1-\tau)}{2} n^{-1} - \sum_{j = 2}^{m-1}\frac{(-\tau)^{j}}{j!} n^{-(j-1)(1-\tau)}
		+\calO\big(n^{-2+\tau}\big)
		+\calO\big(n^{-(m-1)(1-\tau)}\big).
	\]
	In particular, the assumptions of Theorem \ref{thm:4} are satisfied.
\end{example}

For a given sequence $(\gamma_n : n \in \NN_0)$, the following proposition provides an explicit sequence 
$(\tilde{a}_n : n \in \NN_0)$ satisfying the regularity assumptions of Theorem \ref{thm:4}.
\begin{proposition}
	Suppose that $(\gamma_n : n \in \NN)$ is a positive sequence such that
	\[
		\lim_{n \to \infty} \gamma_n = \infty,
		\qquad\text{and}\qquad
		\bigg( \frac{1}{\gamma_n} : n \in \NN \bigg) \in \calD_1^N(\RR; w)
	\]
	where $w=(w_n : n \in \NN)$ is a weight. For $\kappa > 0$ we set
	\[
		\tilde{a}_n = \exp \bigg(\sum_{j=1}^n \frac{\kappa}{\gamma_j} \bigg).
	\]
	Then
	\[
		\lim_{n \to \infty} \gamma_n \Big( 1 - \frac{\tilde{a}_{n-1}}{\tilde{a}_n} \Big) = \kappa,
	\]
	and
	\[
		\bigg( \gamma_n \Big( 1 - \frac{\tilde{a}_{n-1}}{\tilde{a}_n} \Big) : n \in \NN \bigg) \in \calD_1^N(\RR; w).
	\]
\end{proposition}
\begin{proof}
	We have
	\[
		\gamma_n \Big( 1 - \frac{\tilde{a}_{n-1}}{\tilde{a}_n} \Big) =
		\gamma_n \bigg( 1 - \exp \Big(-\frac{\kappa}{\gamma_n} \Big) \bigg) = 
		f \Big(\frac{1}{\gamma_n} \Big)
	\]
	where
	\[
		f(x) = \frac{1 - \ue^{-\kappa x}}{x}.
	\]
	Observe that
	\[
		\lim_{x \to 0} f(x) = \kappa.
	\]
	Moreover, $f$ has analytic extension to $\RR$, thus by the mean value theorem
	\[
		\Big| f \Big(\frac{1}{\gamma_{n+N}} \Big) - f \Big(\frac{1}{\gamma_n} \Big) \Big| \leq c 
		\Big| \frac{1}{\gamma_{n+N}} - \frac{1}{\gamma_{n}} \Big|,
	\]
	from which the conclusion follows.
\end{proof}

The following proposition settles the problem when the Carleman's condition is satisfied in terms of the growth of
the sequence $(\gamma_n : n \in \NN_0)$.
\begin{proposition} 
	\label{prop:6}
	Suppose that $(\gamma_n : n \in \NN)$ and $(\tilde{a}_n : n \in \NN_0)$ are positive sequences satisfying
	\[
		\lim_{n \to \infty} \gamma_n = \infty, \qquad \text{and} \qquad
		\lim_{n \to \infty} \gamma_n \Big( 1 - \frac{\tilde{a}_{n-1}}{\tilde{a}_n} \Big) = \kappa > 0.
	\]
	Then
	\begin{enumerate}[(i), leftmargin=2em]
		\item
		\label{prop:6:a}
		if $\lim_{n \to \infty} \frac{\gamma_n}{n} = 0$, then $\sum_{n=0}^\infty \frac{1}{\tilde{a}_n} < \infty$;
		\item 
		\label{prop:6:b}
		if $\lim_{n \to \infty} \frac{\gamma_n}{n} = \infty$, then $\sum_{n=0}^\infty \frac{1}{\tilde{a}_n} = \infty$.
	\end{enumerate}
\end{proposition}
\begin{proof}
	We shall prove \ref{prop:6:a} only, as the proof of \ref{prop:6:b} is similar. Let
	\[
		r_n = \gamma_n \Big( 1 - \frac{\tilde{a}_{n-1}}{\tilde{a}_n} \Big).
	\]
	There is $n_0$ such that for $n \geq n_0$,
	\[
		\frac{\gamma_n}{n} \leq \frac{\kappa}{4} \leq \frac{r_n}{2}.
	\]
	Hence, for $j \geq n_0$,
	\[
		\frac{\tilde{a}_{j-1}}{\tilde{a}_j} = 1 - \frac{r_j}{\gamma_j} 
		\leq  1 - \frac{2}{j},
	\]
	and so
	\[
		\frac{\tilde{a}_{n_0-1}}{\tilde{a}_n} 
		= 
		\prod_{j = n_0}^n \frac{\tilde{a}_{j-1}}{\tilde{a}_j}
		\leq
		\prod_{j=n_0}^n \bigg( 1 - \frac{2}{j} \bigg).
	\]
	Consequently, for a certain $c>0$,
	\[
		\frac{\tilde{a}_{n_0-1}}{\tilde{a}_n}
		\leq c n^{-2},
	\]
	which implies that
	\[
		\sum_{n=0}^\infty \frac{1}{\tilde{a}_n} < \infty.\qedhere
	\]
\end{proof}

The following proposition has a proof similar to Proposition \ref{prop:5}.
\begin{proposition}
	Suppose that the hypotheses of Theorem~\ref{thm:4} are satisfied for a sequence $(\gamma_n : n \in \NN_0)$ such that
	\[
		\lim_{n \to \infty} \frac{\gamma_n}{n} = 0.
	\]
	Assume that $\discr \calR_0 > 0$. Then
	\begin{enumerate}[(i), leftmargin=2em]
		\item if $-\kappa + \frac{1}{N} \sqrt{\discr \calR_0} > 0$ then the operator $A$ is self-adjoint;
		\item if $-\kappa + \frac{1}{N} \sqrt{\discr \calR_0} < 0$ then the operator $A$ is not self-adjoint.
	\end{enumerate}
	Moreover, if $A$ is self-adjoint then $\sigmaEss(A) = \emptyset$.
\end{proposition}

\subsubsection{Construction of the modulating sequences}
In this section we present examples of sequences $(\alpha_n : n \in \NN_0)$ and $(\beta_n : n \in \NN_0)$
for which one can compute $\tr \calR_0$ and $\discr \calR_0$.

The first example illustrates that the sign of $\discr \calR_0$ may be positive or negative.
\begin{example} 
	\label{ex:1}
	Let $N=3$, and 
	\[
		\alpha_n \equiv 1, \qquad\text{and}\qquad
		\beta_n \equiv 1.
	\]
	Then $\sigma = 1$ and
	\[
		\calR_0 =
		\begin{pmatrix}
			-\kappa - \frakf_0 + \frakf_2 & \kappa - \frakf_0 + \frakf_1 \\
			-\kappa + \frakf_1 - \frakf_2 & -2 \kappa + \frakf_0 - \frakf_2
		\end{pmatrix}.
	\]
	Consequently,
	\[
		\tr \calR_0 = - 3 \kappa \qquad \text{and} \qquad
		\discr \calR_0 = 
		4 
		\Big( \frakf_0^2 + \frakf_1^2 + \frakf_2^2 - \frakf_0 \frakf_1 - \frakf_0 \frakf_2 - \frakf_1 \frakf_2 \Big) 
		- 3 \kappa^2.
	\]
	In particular, taking $\frakf_0 = \frakf_1 = 0$ and $\frakf_2 = t$, we obtain
	\[
		\sign{\discr \calR_0} =
		\begin{cases}
			1 & |t| > \frac{\sqrt{3}}{2} \kappa, \\
			0 & |t| = \frac{\sqrt{3}}{2} \kappa, \\
		   -1 & |t| < \frac{\sqrt{3}}{2} \kappa.
		\end{cases}
	\]
\end{example}

In the following example, discriminant of $\calR_0$ is non-negative regardless of $(\frakf_n : n \in \ZZ)$.
\begin{example}
	\label{ex:2}
	Let $N=4$, and 
	\[
		\alpha_n \equiv 1, \qquad
		\beta_{n} = 
		\begin{cases}
			(-1)^{n/2} & \text{$n$ even}, \\
			0 & \text{otherwise.}
		\end{cases}
	\]
	Then $\sigma = 1$ and
	\[
		\calR_0 =
		\begin{pmatrix}
			-2 \kappa - \frakf_0 + \frakf_1 - \frakf_2 + \frakf_3 & -\frakf_0 + 2 \frakf_1 - \frakf_2 \\
			0 & -2 \kappa + \frakf_0 - \frakf_1 + \frakf_2 - \frakf_3
		\end{pmatrix}.
	\]
	Consequently,
	\[
		\tr \calR_0 = -4 \kappa \qquad \text{and} \qquad
		\discr \calR_0 =
		4 \bigg( \sum_{j=0}^3 (-1)^j \frakf_j \bigg)^2 \geq 0.
	\]
\end{example}

The following theorem provides a large class of modulating sequences for which $\discr \calR_0$ is always non-negative.
\begin{theorem}
	\label{thm:5}
	Let $N$ be an even integer and $\kappa > 0$. Let $(\frakf_n : n \in \ZZ)$ be $N$-periodic sequence of non-negative
	numbers and $(\alpha_n : n \in \ZZ)$ be $N$-periodic sequence of positive numbers satisfying
	\begin{equation}
		\label{eq:55}
		\alpha_0 \alpha_2 \cdots \alpha_{N-2} = \alpha_1 \alpha_3 \cdots \alpha_{N-1}.
	\end{equation}
	Let $\frakB_n$ denote the transfer matrix associated with sequences $(\alpha_n : n \in \ZZ)$
	and $\beta_n \equiv 0$. We set
	\[
		\calR_0 
		=
		\sum_{j=0}^{N-1} 
		\frac{\alpha_{j-1}}{\alpha_{j}} \big( \kappa + \frakf_{j} - \frakf_{j-1} \big)
		\left\{ \prod_{m=j+1}^{N-1} \frakB_{m}(0) \right\}
		\begin{pmatrix}
			0 & 0 \\
			1 & 0
		\end{pmatrix}
		\left\{ \prod_{m=0}^{j-1} \frakB_{m}(0) \right\}.
	\]
	Then 
	\[
		\tr \calR_0 = -(-1)^{N/2} N \kappa \qquad \text{and} \qquad
		\discr \calR_0 = 4 \bigg( \sum_{j=0}^{N-1} (-1)^j \frakf_j \bigg)^2.
	\]
\end{theorem}
\begin{proof}
	Let $N = 2M$. By \cite[Proposition 3]{PeriodicIII}, for all $\ell \geq k \geq 0$ we have
	\begin{equation} 
		\label{eq:53}
		\prod_{m=k}^{\ell} \frakB_m(0) =
		\begin{pmatrix}
			-\frac{\alpha_{k-1}}{\alpha_k} \frakp^{[k+1]}_{\ell-k-1}(0) & \frakp^{[k]}_{\ell-k}(0) \\
			-\frac{\alpha_{k-1}}{\alpha_k} \frakp^{[k+1]}_{\ell-k}(0) & \frakp^{[k]}_{\ell-k+1}(0)
		\end{pmatrix}.
	\end{equation}
	Observe that for $k \geq 1$ and $j \geq 0$,
	\begin{align*}
		\prod_{m=j}^{j+2k-1} \frakB_m(0) &= 
		\prod_{m=0}^{k-1} \Big( \frakB_{j+2m+1}(0) \frakB_{j+2m}(0) \Big) 
		= 
		\prod_{m=0}^{k-1} 
		\begin{pmatrix}
			-\frac{\alpha_{j+2m-1}}{\alpha_{j+2m}} & 0 \\
			0 & -\frac{\alpha_{j+2m}}{\alpha_{j+2m+1}}
		\end{pmatrix} \\&=
		(-1)^k
		\begin{pmatrix}
			\frac{\alpha_{j+2k-3}}{\alpha_{j+2k-2}} \ldots 
			\frac{\alpha_{j+1}}{\alpha_{j+2}} \frac{\alpha_{j-1}}{\alpha_j} & 0 \\
			0 & \frac{\alpha_{j+2k-2}}{\alpha_{j+2k-1}} \ldots 
			\frac{\alpha_{j+2}}{\alpha_{j+3}} \frac{\alpha_{j}}{\alpha_{j+2}}
		\end{pmatrix}.
	\end{align*}
	In particular, by \eqref{eq:55}, we obtain
	\[
		\prod_{m=0}^{N-1} \frakB_m(0) = (-1)^M \Id.
	\]
	Moreover, by \eqref{eq:53}, for all $j \geq 0$ and $n \geq 0$,
	\begin{equation}
		\label{eq:54}
		\frakp^{[j]}_{n}(0) =
		\begin{cases}
			(-1)^k \frac{\alpha_{j+2k-2}}{\alpha_{j+2k-1}} \ldots 
			\frac{\alpha_{j+2}}{\alpha_{j+3}} \frac{\alpha_{j}}{\alpha_{j+2}} & n=2k, \\
			0 & \text{otherwise.}
		\end{cases}
	\end{equation}
	Setting
	\[
		s_j = \kappa + \frakf_j - \frakf_{j-1},
	\]
	by \eqref{eq:31}, we write
	\[
		\calR_0 =
		\sum_{j=0}^{N-1} \frac{\alpha_{j-1}}{\alpha_j} s_j
		\left\{ \prod_{m=j+1}^{N-1} \frakB_{m}(0) \right\}
		\begin{pmatrix}
			0 & 0 \\
			1 & 0
		\end{pmatrix}
		\left\{ \prod_{m=0}^{j-1} \frakB_{m}(0) \right\}.
	\]
	Therefore, by \eqref{eq:53},
	\[
		\calR_0 =
		\sum_{j=0}^{N-1} \frac{\alpha_{j-1}}{\alpha_j} s_j
		\begin{pmatrix}
			-\frac{\alpha_{j}}{\alpha_{j+1}} \frakp^{[j+2]}_{N-j-3}(0) &
			\frakp^{[j+1]}_{N-j-2}(0) \\
			-\frac{\alpha_{j}}{\alpha_{j+1}} \frakp^{[j+2]}_{N-j-2}(0) &
			\frakp^{[j+1]}_{N-j-1}(0)
		\end{pmatrix}
		\begin{pmatrix}
			0 & 0 \\
			1 & 0
		\end{pmatrix}
		\begin{pmatrix}
			-\frac{\alpha_{N-1}}{\alpha_0} \frakp^{[1]}_{j-2}(0) &
			\frakp^{[0]}_{j-1}(0) \\
			-\frac{\alpha_{N-1}}{\alpha_0} \frakp^{[1]}_{j-1}(0) &
			\frakp^{[0]}_{j}(0)
		\end{pmatrix},
	\]
	and consequently,
	\[
		\calR_0 =
		\sum_{j=0}^{N-1} \frac{\alpha_{j-1}}{\alpha_j} s_j
		\begin{pmatrix}
			-\frac{\alpha_{N-1}}{\alpha_0} \frakp^{[1]}_{j-2}(0) \frakp^{[j+1]}_{N-j-2}(0) & 
			\frakp^{[j+1]}_{N-j-2}(0) \frakp^{[0]}_{j-1}(0) \\
			-\frac{\alpha_{N-1}}{\alpha_0} \frakp^{[j+1]}_{N-j-1}(0) \frakp^{[1]}_{j-2}(0) & 
			\frakp^{[j+1]}_{N-j-1}(0) \frakp^{[0]}_{j-1}(0)
		\end{pmatrix}.
	\]
	In view of \eqref{eq:54}, we have
	\[
		\calR_0 =
		\sum_{j=0}^{N-1} \frac{\alpha_{j-1}}{\alpha_j} s_j
		\begin{pmatrix}
			-\frac{\alpha_{N-1}}{\alpha_0} \frakp^{[1]}_{j-2}(0) \frakp^{[j+1]}_{N-j-2}(0) & 
			0 \\
			0 & 
			\frakp^{[j+1]}_{N-j-1}(0) \frakp^{[0]}_{j-1}(0)
		\end{pmatrix}.
	\]
	By considering even and odd $j$, the last formula can be written in the form
	\begin{align*}
		\calR_0 &=
		\sum_{k=0}^{M-1} \frac{\alpha_{2k-1}}{\alpha_{2k}} s_{2k}
		\begin{pmatrix}
			-\frac{\alpha_{N-1}}{\alpha_0} \frakp^{[1]}_{2k-2}(0) \frakp^{[2k+1]}_{N-2k-2}(0) & 0 \\
			0 & 0
		\end{pmatrix} \\
		&\phantom{=}+
		\sum_{k=0}^{M-1} \frac{\alpha_{2k}}{\alpha_{2k+1}} s_{2k+1}
		\begin{pmatrix}
			0 & 0 \\
			0 & \frakp^{[2k+2]}_{N-2k-2}(0) \frakp^{[0]}_{2k}(0)
		\end{pmatrix}.
	\end{align*}
	Now, using \eqref{eq:54} and \eqref{eq:55} we obtain
	\[
		\frakp^{[2k+2]}_{N-2k-2}(0) \frakp^{[0]}_{2k}(0) = 
		(-1)^{M-1} \frac{\alpha_0 \alpha_2 \ldots \alpha_{N-2}}{\alpha_{1} \alpha_3 \ldots \alpha_{N-1}} \cdot 
		\frac{\alpha_{2k+1}}{\alpha_{2k}} =
		(-1)^{M-1} \cdot 1 \cdot \frac{\alpha_{2k+1}}{\alpha_{2k}}.
	\]
	Analogously one can show
	\[
		-\frac{\alpha_{N-1}}{\alpha_0} \frakp^{[1]}_{2k-2}(0) \frakp^{[2k+1]}_{N-2k-2}(0) 
		= (-1)^{M-1} \frac{\alpha_{2k}}{\alpha_{2k-1}}.
	\]
	Therefore,
	\begin{align*}
		\calR_0 &= (-1)^{M-1}
		\begin{pmatrix}
			\sum_{k=0}^{M-1} s_{2k}  & 0 \\
			0 & \sum_{k=0}^{M-1} s_{2k+1} 
		\end{pmatrix} \\
		&=
		-\sigma
		\begin{pmatrix}
			M \kappa + \sum_{k=0}^{M-1} \big( \frakf_{2k} - \frakf_{2k-1} \big) & 0 \\
			0 & M \kappa + \sum_{k=0}^{M-1} \big( \frakf_{2k+1} - \frakf_{2k} \big)
		\end{pmatrix},
	\end{align*}
	and the conclusion readily follows.
\end{proof}

\begin{bibliography}{jacobi}
	\bibliographystyle{amsplain}
\end{bibliography}

\end{document}